\documentclass[11pt]{amsart}
\usepackage{amscd,amssymb,amsthm,amsmath,amssymb,mathrsfs,enumerate,epigraph}
\usepackage[matrix,arrow,curve]{xy}
\usepackage[margin=2cm]{geometry}
\usepackage{mathtools}

\newtheorem{theorem}{Theorem}[section]

\newtheorem{lemma}[theorem]{Lemma}
\newtheorem*{lemma*}{Lemma}
\newtheorem{corollary}[theorem]{Corollary}

\newtheorem*{corollary*}{Corollary}

\theoremstyle{definition}
\newtheorem{example}[theorem]{Example}
\newtheorem*{example*}{Example}

\theoremstyle{remark}
\newtheorem{remark}[theorem]{Remark}

\makeatletter\@addtoreset{equation}{section} \makeatother

\thanks{The title is inspired by the~book ``The World of Null-A'', A.\ E.\ van Vogt,  	Simon \& Schuster, 1948}

\title{K-stability of Fano 3-folds in the~World of Null-A}

\author{Hamid Abban, Ivan Cheltsov, Takashi Kishimoto, Fr\'ed\'eric Mangolte}

\let\origmaketitle\maketitle
\def\maketitle{
  \begingroup
  \def\uppercasenonmath##1{} 
  \let\MakeUppercase\relax 
  \origmaketitle
  \endgroup
}

\begin{document}

\begin{abstract} A variety is said to satisfy Condition $(\mathbf{A})$ if every finite abelian subgroup of its automorphism group has a fixed point.
We show that a smooth Fano 3-fold  not satisfying Condition~$(\mathbf{A})$ is K-polystable unless it is contained in eight exceptional deformation families (seven of them consists of one smooth member, and one of them has one-parameter moduli).
\end{abstract}

\address{\emph{Hamid Abban}\newline
\textnormal{University of Nottingham, Nottingham, England
\newline
\texttt{hamid.abban@nottingham.ac.uk}}}

\address{ \emph{Ivan Cheltsov}\newline
\textnormal{University of Edinburgh, Edinburgh, Scotland
\newline
\texttt{i.cheltsov@ed.ac.uk}}}

\address{\emph{Takashi Kishimoto}\newline
\textnormal{Saitama University, Saitama, Japan
\newline
\texttt{kisimoto.takasi@gmail.com}}}

\address{\emph{Fr\'ed\'eric Mangolte}\newline
\textnormal{Aix Marseille University, CNRS, I2M, Marseille, France
\newline
\texttt{frederic.mangolte@univ-amu.fr}}}

\maketitle


\section{Introduction}
\label{sec:introduction}

\setlength{\epigraphwidth}{0.6\textwidth}
\epigraph{The key to understanding the world is to see beyond the apparent and recognize the underlying patterns and connections.}{\textit{The World of Null-A}, A.\ E.\ van Vogt, 1948}

The study of Fano varieties in algebraic geometry offers a rich interplay between geometric structure, arithmetic properties, and birational classification. They are characterized by their ample anticanonical divisors, a property that endows them with a wealth of positivity and connects them to fundamental questions in the~minimal model program and moduli theory. Among Fano varieties, smooth Fano 3-folds hold a privileged position due to their comprehensive classification into 105 deformation families, a monumental achievement by Iskovskikh, Mori, and Mukai. This classification provides a detailed roadmap for exploring their properties, from their Picard groups to their automorphism groups, making them an ideal testing ground for conjectures and theories in modern geometry.

In recent years, the~concept of K-stability has emerged as a transformative invariant in the~study of Fano varieties. Rooted in the~work of Tian, Donaldson, and others, K-stability serves as a bridge between algebraic geometry and differential geometry, determining the~existence of K\"ahler-Einstein metrics on Fano manifolds. Beyond its metric implications, K-stability has profound connections to the~construction of moduli spaces, offering a stability condition that governs the~compactification of families of Fano varieties. This development has spurred a flurry of research, uncovering surprising links between K-stability and other invariants, such as the~existence of rational points in arithmetic settings or the~behavior of automorphism groups in the geometric setting. Our paper builds on this momentum, seeking to deepen these connections by exploring a purely geometric analogue to an arithmetic phenomenon previously identified in the~context of smooth Fano 3-folds.

The inspiration for this work stems from a recent result established in prior research~\cite{AbbanCheltsovKishimotoMangolte}, which examined smooth Fano 3-folds defined over subfields $\Bbbk\subset \mathbb{C}$. That study revealed that if the~geometric model of such a 3-fold is not K-polystable, then the~variety typically possesses a $\Bbbk$-rational point, barring a small set of exceptional cases. This finding suggested an intimate relationship between K-polystability and the~arithmetic property of ``pointlessness'' (the absence of rational points), raising the~question of whether a similar principle might hold in a geometric framework, independent of field considerations. To pursue this, we shift our focus from rational points to the~action of finite abelian groups on the~variety, introducing a new geometric condition that mirrors the~arithmetic one.

Following \cite{CheltsovTschinkelZhang2025},
we define Condition~$\mathbf{(A)}$ for a smooth variety $X$ as follows: $X$ satisfies Condition~$\mathbf{(A)}$ if every finite abelian subgroup of its automorphism group
$\operatorname{Aut}(X)$ fixes at least one point in $X$. This definition draws inspiration from the~Lang--Nishimura theorem and Koll\'ar--Szab\'o's birational invariance results,
which highlight the~significance of fixed points in equivariant geometry. Our primary objective is to classify non-K-polystable smooth Fano 3-folds according to whether
they satisfy Condition~$\mathbf{(A)}$, and to elucidate the~exceptional cases where this condition fails, drawing parallels to the~arithmetic exceptions identified in~\cite{AbbanCheltsovKishimotoMangolte}.

\subsection{Main Result}
\label{subsection:Main-Theorem}
Our Main Theorem asserts that every smooth Fano 3-fold that is not K-polystable satisfies Condition~$\mathbf{(A)}$, with exactly eight exceptions. These exceptional varieties are:
\begin{enumerate}
    \item $\mathbb{P}^1 \times \mathbb{F}_1$, where $\mathbb{F}_1$ is the~first Hirzebruch surface,
    \item $\mathbb{P}^1 \times S$, where $S$ is a smooth del Pezzo surface of degree 7,
    \item The blowup of a smooth quadric in $\mathbb{P}^4$ along a quartic elliptic curve,
    \item The blowup of a quadric cone in $\mathbb{P}^4$ at its vertex,
    \item The blowup of $\mathbb{P}^1 \times \mathbb{P}^1 \times \mathbb{P}^1$ along a smooth curve of degree $(0,1,1)$,
    \item The blowup of $\mathbb{P}^3$ along a line,
    \item The blowup of $\mathbb{P}^3$ at two distinct points followed by the~blowup of the~strict transform of the~line through these points,
    \item The blowup of $\mathbb{P}^3$ at two distinct points followed by the~blowup of the~strict transforms of two disjoint lines, one passing through the~points and another disjoint from it.
\end{enumerate}

For each of these exceptions, we explicitly construct finite abelian subgroups of their automorphism groups that act without fixed points, confirming their failure to satisfy Condition~$\mathbf{(A)}$. Moreover, we observe that these varieties are not K-polystable, and most are rigid within their deformation families, lacking moduli.
A key corollary follows: if a smooth Fano 3-fold is strictly K-semistable, then it satisfies Condition~$\mathbf{(A)}$, as all such varieties fall outside the~exceptional list.

The parallels between these geometric exceptions and the~arithmetic exceptions from~\cite{AbbanCheltsovKishimotoMangolte} are striking. For instance, products like $\mathbb{P}^1 \times \mathbb{F}_1$ and certain blowups appear in both settings, suggesting a deeper unity between the~arithmetic and geometric manifestations of K-stability. This convergence reinforces the~hypothesis that K-stability governs fundamental properties of Fano 3-folds, whether viewed through the~lens of rational points or group actions.

\subsection{Plan of the~proof}

Recall that smooth Fano 3-folds have been classified into 105 deformation families by Iskovskikh, Mori and Mukai  \cite{IsPr99}.
In this paper, we follow the~Mori--Mukai numbering of the~105 families, written as ``Family \textnumero $m.n$'', in which $m$ is the~rank of the~Picard group of the~3-fold, ranging from 1 to 10, and $n$ is simply a list number. To achieve our main result, we partition 105 deformation families of smooth Fano 3-folds into the~following three sets:
\begin{enumerate}[(i)]
\item $53$ families in which all smooth elements are K-polystable \cite{AbbanZhuangSeshadri,Book,BelousovLoginov,CheltsovDenisovaFujita,CheltsovFedorchukFujitaKaloghiros,CheltsovFujitaKishimotoOkada,CheltsovFujitaKishimotoPark,LTN,Liu,Zhao};
\item $27$ families where all smooth members are not K-polystable \cite{Book,Fujita2019};
\item $25$ families in which only general members are known to be K-polystable \cite{Book}.
\end{enumerate}
The first $53$ deformation families are
\begin{center}
\textnumero 1.1, \textnumero  1.2, \textnumero 1.3, \textnumero 1.4, \textnumero 1.5, \textnumero 1.6, \textnumero 1.7, \textnumero 1.8, \textnumero 1.11, \textnumero 1.12, \textnumero 1.13, \textnumero 1.14, \textnumero 1.15,\\
\textnumero 1.16, \textnumero 1.17, \textnumero 2.1, \textnumero 2.2, \textnumero 2.3, \textnumero 2.4, \textnumero 2.6, \textnumero 2.7, \textnumero 2.25, \textnumero 2.8, \textnumero 2.15, \textnumero 2.18, \textnumero 2.19, \\
\textnumero 2.27, \textnumero 2.29, \textnumero 2.32, \textnumero 2.34, \textnumero 3.1, \textnumero 3.3, \textnumero 3.4, \textnumero 3.9, \textnumero 3.15, \textnumero 3.17, \textnumero 3.19, \textnumero 3.20, \textnumero 3.25, \\
\textnumero 3.27, \textnumero 4.1, \textnumero 4.2, \textnumero 4.3, \textnumero 4.4, \textnumero 4.6, \textnumero 4.7, \textnumero 5.1, \textnumero 5.3, \textnumero 6.1, \textnumero 7.1, \textnumero 8.1, \textnumero 9.1, \textnumero 10.1.
\end{center}
These families are irrelevant for us --- we listed them for completeness of exposition.
Similarly, it follows from \cite{Fujita2019} that every smooth member of the~following $26$ families is K-unstable:
\begin{center}
\textnumero 2.23, \textnumero 2.28, \textnumero 2.30, \textnumero 2.31, \textnumero 2.33, \textnumero 2.35, \textnumero 2.36, \textnumero 3.14, \textnumero 3.16, \textnumero 3.18, \textnumero 3.21, \textnumero 3.22, \textnumero 3.23, \\
\textnumero 3.24, \textnumero 3.26, \textnumero 3.28, \textnumero 3.29, \textnumero 3.30, \textnumero 3.31, \textnumero 4.5,  \textnumero 4.8,  \textnumero 4.9,  \textnumero 4.10, \textnumero 4.11, \textnumero 4.12, \textnumero 5.2.
\end{center}
Furthermore, it was shown in \cite{Book} that Family \textnumero 2.26 contains exactly two smooth members and one of them is K-unstable, while the other smooth member is strictly K-semistable.
In Section~\ref{Section:k-point} we show that each element in $19$ of these families satisfy Condition~$(\mathbf{A})$,
hence producing the~$8$ families appearing in Section~\ref{subsection:Main-Theorem} as exceptional cases.
The remaining families are
\begin{center}
\textnumero 1.9,  \textnumero 1.10, \textnumero 2.5,  \textnumero 2.9,  \textnumero 2.10, \textnumero 2.11, \textnumero 2.12, \textnumero 2.13, \textnumero 2.14, \textnumero 2.16, \textnumero 2.17, \textnumero 2.20, \textnumero 2.21, \\
\textnumero 2.22, \textnumero 2.24, \textnumero 3.2,  \textnumero 3.5, \textnumero 3.6,  \textnumero 3.7,  \textnumero 3.8,  \textnumero 3.10, \textnumero 3.11, \textnumero 3.12, \textnumero 3.13, \textnumero 4.13.
\end{center}
Among these we treat Families \textnumero 1.10, \textnumero 2.9, \textnumero 2.11, \textnumero 2.13, \textnumero 2.14, \textnumero 2.17, \textnumero 2.20, \textnumero 2.22, \textnumero 3.8, \textnumero 3.11 in Section~\ref{Section:k-point} by showing that every smooth member satisfy Condition~$(\mathbf{A})$,
and for the~remaining $15$ families we prove in Section~\ref{section:pointless-3-folds} that all smooth members that do not satisfy Condition~$(\mathbf{A})$ are K-polystable.
Note that Families \textnumero 2.21, \textnumero 2.24,  \textnumero 3.5,  \textnumero 3.8,  \textnumero 3.10,  \textnumero 3.12,  \textnumero 3.13, \textnumero 4.13
have non-K-polystable members.

Throughout this paper, all varieties are projective, normal, irreducible, and defined over the~field of complex numbers unless otherwise stated.

\section{Preliminary results}
\label{section:preliminaries}

\setlength{\epigraphwidth}{0.6\textwidth}
\epigraph{True intelligence lies not in knowledge alone, but in the ability to apply that knowledge in creative and innovative ways.}{\textit{The World of Null-A}, A.\ E.\ van Vogt, 1948}

In this section, we collect some technical results that we will be using in this article. Most of these results are known to experts, so the reader may skip them and only consult them as they are referred to. We start with the~following well-known result by J\'anos Koll\'ar and Endre Szab\'o.

\begin{theorem}[{\cite[Proposition~A.4]{ReYou00}}]
\label{theorem:Kollar-Szabo} Let $X$ and $X^\prime$ be smooth
varieties acted on by a finite abelian group $A$. Suppose that
there exists an $A$-equivariant birational map $\psi\colon X\dasharrow X^\prime$.
Then $X$ contains a point fixed by $A$ if and only if $X^\prime$ contains a point fixed by $A$.
\end{theorem}

\begin{corollary}
\label{corollary:Pn-blow-up}
Let $p$ be a point in $\mathbb{P}^n$, and let $\pi\colon X\to\mathbb{P}^n$ be the blowup of this point, where $n\geqslant 2$.
Then every finite abelian subgroup in $\mathrm{Aut}(X)$ fixes a point in $X$.
\end{corollary}

\begin{proof}
Since $\pi$ is $\mathrm{Aut}(X)$-equivariant, the~required assertion follows from Theorem~\ref{theorem:Kollar-Szabo}.
\end{proof}

\begin{remark}
\label{remark:cyclic-Lefschetz}
If $X$ is a (projective) smooth rationally connected variety, then every finite cyclic group acting on $X$ has a fixed point by the~holomorphic Lefschetz fixed point formula.
\end{remark}

The following lemma is a geometric counter-part of \cite[Lemma~2.2]{AbbanCheltsovKishimotoMangolte},
and its special cases are well-known to experts \cite{Alvaro}.
Its proof was communicated to us by Alex Duncan.

\begin{lemma}
\label{lemma:lift}
Let $G$ be a finite subgroup in $\mathrm{PGL}_{n+1}(\mathbb{C})$
such that there is a $G$-invariant hypersurface $X\subset \mathbb{P}^n$ of degree $d$.
Suppose that the~integers $d$ and $n+1$ are coprime.
Then there exists a subgroup $\widetilde{G}\subset\mathrm{GL}_{n+1}(\mathbb{C})$ such that $\phi(\widetilde{G})=G\simeq\widetilde{G}$, where
$\phi\colon\mathrm{GL}_{n+1}(\mathbb{C})\to \mathrm{PGL}_{n+1}(\mathbb{C})$ is the~natural projection.
\end{lemma}

\begin{proof}
Recall that the~restriction of $\phi$ to the~subgroup
$\mathrm{SL}_{n+1}(\mathbb{C})\subset\mathrm{GL}_{n+1}(\mathbb{C})$ gives a surjective homomorphism $\mathrm{SL}_{n+1}(\mathbb{C})\to\mathrm{PGL}_{n+1}(\mathbb{C})$ whose kernel consists of scalar matrices and is isomorphic to $\mathbb{Z}/(n+1)\mathbb{Z}$.
Let $\widehat{G}$ be the~preimage of the~group $G$ in $\mathrm{SL}_{n+1}(\mathbb{C})$ via this epimorphism,
and let $f\in\mathbb{C}[x_1,\ldots,x_{n+1}]$ be the~defining polynomial of the~hypersurface $X$.
Then $f$ is a polynomial of degree $d$, and there exists a group homomorphism $\eta\colon\widehat{G}\to\mathbb{C}^\ast$ such that
$$
g^*(f)=\eta(g)f
$$
for every matrix $g\in\widehat{G}$. In particular, if $g=\lambda I_{n+1}\in\widehat{G}$ for some $\lambda\in\mathbb{C}^\ast$, then $\eta(g)=\lambda^d$. Here, we denote by $I_{n+1}$ the~identity $(n+1)\times(n+1)$ matrix.
Let $\mathrm{det}\colon\mathrm{GL}_{n+1}(\mathbb{C})\to \mathbb{C}^\ast$ be the~group homomorphism that maps every matrix to its determinant.
Then $\mathrm{det}\big(\lambda I_{n+1}\big)=\lambda^{n+1}$
for every $\lambda\in\mathbb{C}^\ast$. Since $d$ and $n+1$ are coprime, we can use appropriate powers of $\eta$ and $\mathrm{det}$ to construct a homomorphism $\theta\colon\widehat{G}\to\mathbb{C}^\ast$
such that $\theta\big(\lambda I_{n+1}\big)=\frac{1}{\lambda}$
for every $\lambda\in\mathbb{C}^\ast$ such that the~scalar matrix $\lambda I_{n+1}$ is contained in $\widehat{G}$.
Now, we let $\widetilde{G}$ be the~subgroup in $\mathrm{GL}_{n+1}(\mathbb{C})$ defined as follows:
$$
\widetilde{G}=\big\{\theta(g)g\ \big\vert\ g\in \widehat{G}\big\}.
$$
Then $\widetilde{G}$ is the~required subgroup in $\mathrm{GL}_{n+1}(\mathbb{C})$.
\end{proof}

\begin{corollary}
\label{corollary:odd-degree-curves}
Let $A$ be a finite abelian subgroup in $\mathrm{PGL}_{n+1}(\mathbb{C})$. Suppose that there is an $A$-invariant irreducible rational curve $C\subset\mathbb{P}^n$ such that $\mathrm{deg}(C)$ is odd,
and there exists a finite subgroup $\widetilde{A}\subset\mathrm{GL}_{n+1}(\mathbb{C})$ such that $\phi(\widetilde{A})=A\simeq\widetilde{A}$, where
$\phi\colon\mathrm{GL}_{n+1}(\mathbb{C})\to \mathrm{PGL}_{n+1}(\mathbb{C})$ is the~projection.
Then $A$ fixes a point in~$C$.
\end{corollary}

\begin{proof}
Let $\nu\colon\widehat{C}\to C$ be the~normalization of the~curve $C$. Then $\nu$ is $A$-equivariant, and we have $\widehat{C}\simeq\mathbb{P}^1$.
On the~other hand, since $\widetilde{A}$ is abelian, there is an $A$-invariant hyperplane $H$ in $\mathbb{P}^{n}$ that does not contain~$C$,
so $\nu^*(H\vert_{C})$ is an $A$-invariant divisor on $\widehat{C}\simeq\mathbb{P}^1$ of degree $\mathrm{deg}(C)$.
Now, Lemma~\ref{lemma:lift} implies that $A$ fixes a point in $\widehat{C}$, so it also fixes a point in $C$.
\end{proof}

The following lemma is a geometric counter-part of \cite[Lemma~3.4]{AbbanCheltsovKishimotoMangolte}.

\begin{lemma}
\label{lemma:dP-negative}
Let $S$ be a del Pezzo surface with at worst quotient singularities, let $A$ be a finite abelian subgroup in $\mathrm{Aut}(S)$,
and $C$ an $A$-invariant irreducible curve in $S$,
let $f\colon\widetilde{S}\to S$ be the~minimal resolution of singularities,
let $\widetilde{C}$ be the~strict transform on $\widetilde{S}$ of the~curve $C$.
If $\widetilde{C}^2<0$, then $A$ fixes a point in $C$.
\end{lemma}

\begin{proof}
Observe that the~action of the~group $A$ lifts to $\widetilde{S}$.
If $\widetilde{C}^2<0$, then it follows from the~adjunction formula that $\widetilde{C}^2=-1$ and $\widetilde{C}\simeq\mathbb{P}^1$,
so there exists an $A$-equivariant contraction $\widetilde{S}\to\overline{S}$ of the~curve $\widetilde{C}$ to a smooth $A$-fixed point $P$.
Then the curve $\widetilde{C}$ is identified with the~projectivization of the~Zariski tangent space $V=T_P(\overline{S})$.
Since $A$ is abelian, $V$ splits as a sum of
$A$-invariant one-dimensional subspaces, which means that $C$ has
an $A$-fixed point.
\end{proof}

\begin{corollary}
\label{corollary:dP-negative-easy}
Let $S$ be a del Pezzo surface with at worst quotient singularities, let $A$ be a finite abelian subgroup in $\mathrm{Aut}(S)$,
and $C$ an $A$-invariant irreducible curve in $S$. Suppose that $A$ fixes no points in~$S$.
Then $C^2\geqslant 0$. Moreover, if $C^2=0$ then $C$ is contained in the~smooth locus of the~surface $S$.
\end{corollary}

\begin{corollary}
\label{corollary:dP-negative-nice}
Let $S$ be a del Pezzo surface with at worst quotient singularities, let $A$ be a finite abelian subgroup in $\mathrm{Aut}(S)$,
and $\pi\colon S\to S^\prime$ an $A$-equivariant birational morphism such that $S^\prime$ is a normal surface.
Then $A$ fixes a point in $S$ if and only if $A$ fixes a point in $S^\prime$.
\end{corollary}

\begin{proof}
Obviously, if $A$ fixes a point in $S$, then $A$ fixes a point in $S^\prime$. Thus, to complete the~proof, we may assume that $A$ does not fix points in $S$ but $A$ fixes a point $p\in S^\prime$ and seek for a contradiction.

First, we observe that $S^\prime$ also has at worst quotient singularities, and $S^\prime$ is also a del Pezzo surface. Applying relative $A$-equivariant Minimal Model Program to $S$ over $S^\prime$, we may assume that
$$
\mathrm{rk}\big(\mathrm{Pic}^A(S)\big)=\mathrm{rk}\big(\mathrm{Pic}^A(S^\prime)\big)+1,
$$
and all $\pi$-exceptional irreducible curves are mapped to $p$.
Denote these curves by $E_1,\ldots,E_r$. If $r=1$, then we get a contradiction by Corollary~\ref{corollary:dP-negative-easy}, since $E_1^2<0$. Suppose $r\geqslant 2$ and consider the~following $A$-equivariant commutative diagram:
$$
\xymatrix@R=1em{
&&\widetilde{S}\ar@{->}[dll]_{\phi}\ar@{->}[drr]^{\varphi}&&\\%
S\ar@{->}[drr]_{\pi}&&&&\widehat{S}\ar@{->}[dll]^{\varpi}\\%
&&S^\prime&& }
$$
where $\varpi$ is a minimal resolution of $p\in S^\prime$,
$\phi$ is the~minimal resolution of singularities of singular points of $S$ mapped to $p$,
and $\varphi$ is a birational morphism. Then $A$ does not fix points in $\widetilde{S}$,
so it follows from Theorem~\ref{theorem:Kollar-Szabo} that $A$ also does not fix points in $\widehat{S}$.

Denote by $F_1,\ldots,F_k$ the~irreducible exceptional curves of the~morphism $\varpi$, which are all smooth and isomorphic to $\mathbb{P}^1$. Every two intersecting curves among them meet transversally at one point,
and the~dual graph of these curves does not contain cycles.
All possibilities for the~dual graphs of the~curves $F_1,\ldots,F_k$ are given in \cite{Iliev}.
Since the~dual graph is connected and $A$ does not fix points in $\widehat{S}$,
it immediately follows from Iliev's classification \cite{Iliev} that $A$ leaves invariant at least one curve among $F_1,\ldots,F_k$,
which also follows directly from Nadel's \cite[Theorem 4.5]{Nadel}.

Without loss of generality, we may assume that $F_1$ is $A$-invariant.
Let $\widetilde{F}_1$ be the~strict transform of this curve on $\widetilde{S}$.
Then $\widetilde{F}_1$ is not $\phi$-exceptional, since otherwise $\phi(\widetilde{F}_1)$ would be fixed by $A$.
Hence, we conclude that $\phi(\widetilde{F}_1)$ is one of the~$\pi$-exceptional curves $E_1,\ldots,E_r$.
Now, applying Corollary~\ref{corollary:dP-negative-easy} to $\phi(\widetilde{F}_1)$, we obtain a contradiction.
\end{proof}

Lemma~\ref{lemma:dP-negative} and its corollaries imply the~following result.

\begin{lemma}
\label{lemma:dP}
Let $S$ be a del Pezzo surface with at worst quotient singularities, let $A$ be a finite abelian subgroup in $\mathrm{Aut}(S)$ such that $A$ does not fix points in $S$,
let $C$ be an irreducible $A$-invariant curve in $S$, and let $\Delta$ be an effective $\mathbb{Q}$-divisor on $S$ with $\tau C+\Delta\sim_{\mathbb{Q}} -K_{S}$
for some $\tau\in\mathbb{Q}_{>0}$. Then $\tau\leqslant 2$.
\end{lemma}

\begin{proof}
Suppose that $\tau>2$. First, let us apply $A$-equivariant Minimal Model Program to the~surface $S$.
If this gives an $A$-equivariant birational morphism $\pi\colon S\to S^\prime$, then $S^\prime$ is a del Pezzo surface with at worst quotient singularities such that $\mathrm{rk}\big(\mathrm{Pic}^A(S)\big)>\mathrm{rk}\big(\mathrm{Pic}^A(S^\prime)\big)$.
Moreover, it follows from Corollary~\ref{corollary:dP-negative-nice} that $A$ does not fix points in $S^\prime$,
and $C$ is not $\pi$-exceptional by Corollary~\ref{corollary:dP-negative-easy}, so that
$-K_{S^\prime}\sim_{\mathbb{Q}}\tau C^\prime+\Delta^\prime$, where $C^\prime=\pi(C)$ and $\Delta^\prime=\pi(\Delta)$.
Therefore, replacing $S$ with $S^\prime$, we may assume that the~$A$-equivariant Minimal Model Program applied to $S$ does not give a birational morphism.
Then
\begin{itemize}
\item either $\mathrm{rk}(\mathrm{Pic}^A(S))=1$;
\item or $\mathrm{rk}(\mathrm{Pic}^A(S))=2$, and there exist two $A$-equivariant conic bundles $\eta_1\colon S\to\mathbb{P}^1$ and $\eta_2\colon S\to\mathbb{P}^1$ whose general fibers have non-trivial intersection.
\end{itemize}
In the~latter case, we obtain a contradiction by intersecting $\tau C+\Delta\sim_{\mathbb{Q}} -K_S$ with general fibers of the~conic bundles $\eta_1$ and $\eta_2$,
so we have $\mathrm{rk}(\mathrm{Pic}^A(S))=1$.
Increasing $\tau$ if necessary, we may assume that $\Delta=0$, so $-K_{S}\sim_{\mathbb{Q}}\tau C$.
Then $(S,C)$ has purely log terminal singularities.
This follows from \cite[Lemma~2.4.10]{CheltsovShramov2015} as in the~proof of \cite[Lemma 3.1]{AbbanCheltsovKishimotoMangolte}.
In particular, $C$ is smooth \cite{Kawamata1998}.
Moreover, applying Kawamata's subadjunction theorem \cite{Kawamata1998}, we see that $C\simeq\mathbb{P}^1$.

Let $f\colon\widetilde{S}\to S$ be the~minimal resolution of singularities of the~surface $S$,
and let $\widetilde{C}$ be the~strict transform on $\widetilde{S}$ of the~curve $C$.
Then $f$ is $A$-equivariant, $A$ does not fix points in $\widetilde{S}$,
we have $\widetilde{C}^2\geqslant 0$ by Lemma~\ref{lemma:dP-negative},
and it follows from Corollary~\ref{corollary:dP-negative-easy} that either $C^2>0$ or $C^2=0$ and $C$ is contained in the~smooth locus of the~surface $S$.
Moreover, we have $-K_{\widetilde{S}}\sim_{\mathbb{Q}}\tau\widetilde{C}+\widetilde{B}$,
where $\widetilde{B}$ is an effective $\mathbb{Q}$-divisor on  $\widetilde{S}$ whose support consists of the~$f$-exceptional curves.
Now, applying $A$-equivariant Minimal Model Program to $\widetilde{S}$, we obtain an $A$-equivariant birational morphism $h\colon\widetilde{S}\to\overline{S}$ such that $\overline{S}$ is a smooth surface, and one of the~following two cases holds:
\begin{itemize}
\item $\mathrm{rk}(\mathrm{Pic}^A(\overline{S}))=1$ and $\overline{S}$ is a del Pezzo surface;
\item $\mathrm{rk}(\mathrm{Pic}^A(\overline{S}))=2$ and there exists an $A$-equivariant conic bundle $\eta\colon\overline{S}\to\mathbb{P}^1$.
\end{itemize}
Note that $\widetilde{C}$ is not $h$-exceptional, because $\widetilde{C}^2\geqslant 0$.
Set $\overline{C}=h(\widetilde{C})$ and $\overline{B}=h(\widetilde{B})$. Then $-K_{\overline{S}}\sim_{\mathbb{Q}}\tau\overline{C}+\overline{B}$.
Hence, if $\mathrm{rk}(\mathrm{Pic}^A(\overline{S}))=1$, then it follows from \cite{Cheltsov2008} that $\overline{S}\simeq\mathbb{P}^2$ and $\overline{C}$ is a line in $\overline{S}$, which implies that $A$ fixes a point in $\overline{S}$ contradicting Theorem~\ref{theorem:Kollar-Szabo}.
Thus, we conclude that $\mathrm{rk}(\mathrm{Pic}^A(\overline{S}))=2$ and there exists an $A$-equivariant conic bundle $\eta\colon\overline{S}\to\mathbb{P}^1$.
Then, intersecting the~divisor $\tau\overline{C}+\overline{B}$ with a general fiber of the~conic bundle $\eta$, we see that $\overline{C}$ is a fiber of $\eta$.
In particular, $\widetilde{C}^2\leqslant\overline{C}^2=0$, so $\widetilde{C}^2=\overline{C}^2=0$.
This implies that $h$ is an isomorphism in a neighborhood of the~curve $\widetilde{C}$,
and the~complete linear system $|\widetilde{C}|$ gives the~composition morphism $\eta\circ h\colon\widetilde{S}\to\mathbb{P}^1$.
On the~other hand, we have $C^2>0$, since $\mathrm{rk}(\mathrm{Pic}^A(S))=1$.
Thus, we see that $S$ is singular, and the~curve $C$ contains some singular points of the~surface $S$.

Recall that $(S,C)$ has purely log terminal singularities.
Then it follows from \cite{Kawamata1998} that $C$ contains at most three singular points of the~surface $S$,
and all these singular points are cyclic quotient singularities.
Thus, since $C\simeq\mathbb{P}^1$ and $A$ does not fix points in $C$,
it follows from Lemma~\ref{lemma:lift} that $C$ contains two singular points of $S$ forming one $A$-orbit.
Let $P_1$ and $P_2$ be the~points of this orbit.
Then $f$-exceptional curves mapped to the~points $P_1$ and $P_2$ form two disjoint Hirzebruch--Jung strings,
which are swapped by the~action of the~group $A$.
Since $(S,C)$ has purely log terminal singularities,
$\widetilde{C}$ intersects only the~first curves of these strings,
which we denote by $E_1$ and $E_2$.
Then there exists the~following $A$-equivariant commutative diagram:
$$
\xymatrix@R=1em{
&\widetilde{S}\ar@{->}[dl]_{g}\ar@{->}[dr]^{f}&\\%
\widehat{S}\ar@{->}[rr]_{q} && S}
$$
where $g$ is the~contraction of all $f$-exceptional curves except for the~curves $E_1$ and $E_2$,
and $q$ is a partial resolution of singularities of the~surface $S$ that contracts the~strict transforms of the~curves $E_1$ and $E_2$.

Let $\widehat{C}=g(\widetilde{C})$, $\widehat{E}_1=g(E_1)$, $\widehat{E}_2=g(E_2)$.
Then $-K_{\widehat{S}}\sim_{\mathbb{Q}}\tau\widehat{C}+a(\widehat{E}_1+\widehat{E}_2)$ for some  $a\in\mathbb{Q}_{>0}$.
But  $\widehat{C}^2=0$, $\widehat{C}\cdot\widehat{E}_1=\widehat{C}\cdot\widehat{E}_2=1$,
and $-K_{\widehat{S}}\cdot\widehat{C}=2$ by adjunction formula,
which gives $a=1$, because
$$
2=-K_{\widehat{S}}\cdot \widehat{C}=\big(\tau\widehat{C}+a(\widehat{E}_1+\widehat{E}_2)\big)\cdot \widehat{C}=a(\widehat{E}_1+\widehat{E}_2)\cdot \widehat{C}=2a.
$$
Hence, since $\widehat{E}_1$ and $\widehat{E}_2$ are smooth, we have
$-2=\mathrm{deg}\big(K_{\widehat{E}_i}\big)\leqslant\big(K_{\widehat{S}}+\widehat{E}_i\big)\cdot \widehat{E}_i=-\tau\widehat{C}\cdot\widehat{E}_i=-\tau<-2$
by subadjunction formula, which is a contradiction. This completes the~proof of the~lemma.
\end{proof}

Lemma~\ref{lemma:dP} implies the~following geometric counterpart of \cite[Theorem~A]{AbbanCheltsovKishimotoMangolte},
which will be used later in the~proof of the Main Result~\ref{subsection:Main-Theorem}. 

\begin{theorem}
\label{theorem:dP}
Let $S$ be a del Pezzo surface with quotient singularities, and let $A$ be a finite abelian subgroup in $\mathrm{Aut}(S)$.
Suppose that $A$ does not fix points in $S$. Then $S$ is K-polystable.
\end{theorem}

\begin{proof}
Suppose that $S$ is not K-polystable. Then it follows from \cite{Fujita2019,Li2017,Zhuang2021} that $S$ contains an $A$-invariant irreducible curve $C$ with $\beta(C)\leqslant 0$, where
$$
\beta(C)=1-\frac{1}{(-K_S)^2}\int\limits_{0}^{\tau}\mathrm{vol}\big(-K_S-uC\big)du
$$
and $\tau=\mathrm{sup}\big\{u\in\mathbb{R}_{\geqslant 0}\ \big\vert\  -K_S-uC\ \text{is pseudo-effective}\big\}$.
Note that $\tau\in\mathbb{Q}_{>0}$ and $-K_{S}\sim_{\mathbb{Q}}\tau C+\Delta$ for some effective $\mathbb{Q}$-divisor $\Delta$ on the~surface $S$ whose support does not contain $C$.
Then $\tau\leqslant 2$ by Lemma~\ref{lemma:dP}.

Let $f\colon\widetilde{S}\to S$ be the~minimal resolution of singularities of the~surface $S$,
and let $\widetilde{C}$ be the~strict transform on $\widetilde{S}$ of the~curve $C$.
Then $f$ is $A$-equivariant, $A$ does not fix points in $\widetilde{S}$, $\widetilde{C}^2\geqslant 0$ by Lemma~\ref{lemma:dP-negative},
and it follows from Corollary~\ref{corollary:dP-negative-easy} that either $C^2>0$ or $C^2=0$ and $C$ is contained in the~smooth locus of the~surface $S$.
Then, since $\tau\leqslant 2$, it follows from \cite[Lemma~9.7]{Fujita2016} that $C^2=0$, and
$$
-K_S\sim_{\mathbb{Q}}2C+aZ
$$
for some $a\in\mathbb{Q}_{>0}$ and some $A$-irreducible curve $Z\subset S$ such that $Z^2=0$.
In particular, $C$ is contained in the~smooth locus of the~surface $S$.
Hence, by Riemann--Roch formula, the~linear system $|C|$ is a pencil that gives a conic bundle $\eta_1\colon S\to\mathbb{P}^1$.
Since $S$ is a Mori Dream Space, the~linear system $|nZ|$ is base point free for some positive integer $n$,
it also gives a conic bundle $\eta_2\colon S\to\mathbb{P}^1$, and $Z$ is a union of fibers of this conic bundle.
Hence, replacing $Z$ if necessary, we may assume that $Z\cap\mathrm{Sing}(S)=\varnothing$,
and irreducible components of $Z$ are smooth fibers of $\eta_2$.
Write $Z=Z_1+\cdots+Z_r$, where $Z_1,\ldots,Z_r$ are irreducible.
Then, using adjunction formula, we get $2C\cdot Z_1=\big(2C+aZ\big)\cdot Z_1=-K_S\cdot Z_1=2$,
which gives $C\cdot Z_1=1$. Now, taking the~product of $\eta_1$ and $\eta_2$,
we obtain an isomorphism $S\simeq\mathbb{P}^1\times\mathbb{P}^1$, which is a contradiction, since $\mathbb{P}^1\times\mathbb{P}^1$ is K-polystable.
\end{proof}

The following simple lemma is very useful in our arguments.

\begin{lemma}
\label{lemma:Klein-quadric}
Let $Q$ be the~smooth quadric 3-fold in $\mathbb{P}^4$, and let $A$ be a subgroup in $\mathrm{Aut}(Q)$ that is isomorphic to $(\mathbb{Z}/2\mathbb{Z})^2$.
Then $A$ fixes a point in $Q$.
\end{lemma}

\begin{proof}
Note that the~action of $A$ on the~quadric $Q$ is induced by its action on $\mathbb{P}^4$,
and it follows from Lemma~\ref{lemma:lift} that the~action on $\mathbb{P}^4$ is given by some $5$-dimensional representation of the~group $A$,
which must split as a sum of five $1$-dimensional representations.
Since the~group $A$ has exactly four distinct $1$-dimensional representations,
we see that $A$ pointwise fixes a line $\ell\subset\mathbb{P}^4$,
so $A$ fixes a point in  $\ell\cap Q$.
\end{proof}

Let us show how to apply Lemma~\ref{lemma:Klein-quadric} by proving the~following (well-known) result.

\begin{lemma}
\label{lemma:V5}
Let $V_5$ be the~smooth Fano 3-fold with $-K_{V_5}\sim 2H$ and $H^3=5$, where $H$ is an ample Cartier divisor on $V_5$.
Then $X$ satisfies Condition~$\mathbf{(A)}$.
\end{lemma}

\begin{proof}
Let $A$ be a finite abelian subgroup of $\mathrm{Aut}(V_5)$. Suppose that $A$ does not fix points in $V_5$.
Then $A$ is not cyclic by Remark~\ref{remark:cyclic-Lefschetz}.
Moreover, since $\mathrm{Aut}(V_5)\simeq\mathrm{PGL}_{2}(\mathbb{C})$, we conclude that $A\simeq(\mathbb{Z}/2\mathbb{Z})^2$.

Recall $|H|$ gives an embedding $V_5\hookrightarrow\mathbb{P}^6$, so we identify $V_5$ with its image in $\mathbb{P}^6$.
Then the~Hilbert scheme of lines on $V_5$ is isomorphic to $\mathbb{P}^2$,
and the~group $\mathrm{Aut}(V_5)$ acts faithfully on this $\mathbb{P}^2$ leaving invariant a smooth conic
whose points parametrize lines in $V_5$ with normal bundle $\mathcal{O}_{\mathbb{P}^1}(1)\oplus\mathcal{O}_{\mathbb{P}^1}(-1)$.
This implies that $V_5$ contains an $A$-invariant line.
Then, projecting from this line, we obtain an $A$-equivariant birational map $V_5\dasharrow Q$, where $Q$ is a smooth quadric 3-fold.
By Lemma~\ref{lemma:Klein-quadric}, the~group $A$ fixes a point in $Q$, so $A$ fixes a point in $V_5$ by Theorem~\ref{theorem:Kollar-Szabo}.
\end{proof}

We will also need the~following classification based result.

\begin{lemma}
\label{lemma:Prokhorov}
Let $X$ be a smooth Fano 3-fold, and let $G$ be a finite subgroup in $\mathrm{Aut}(X)$.
Suppose that either $\mathrm{Pic}(X)\simeq\mathbb{Z}^2$ and $X$ is not contained in Families \textnumero 2.6, \textnumero 2.12, \textnumero 2.21, \textnumero 2.32,
or $X$ is contained in Families \textnumero 3.5, \textnumero 3.6, \textnumero 3.8, \textnumero 3.11, \textnumero 3.12, \textnumero 3.14, \textnumero 3.16, \textnumero 3.18, \textnumero 3.21, \textnumero 3.22, \textnumero 3.23, \textnumero 3.24, \textnumero 3.26, \textnumero 3.29, \textnumero 3.30, \textnumero 4.5, \textnumero 4.9, \textnumero 4.11.
Then every extremal ray of the~Mori cone $\overline{\mathrm{NE}}(X)$ is $G$-invariant.
\end{lemma}

\begin{proof}
If $\mathrm{Pic}(X)\simeq\mathbb{Z}^2$, the~required assertion follows from \cite[Theorem~1.2]{Prokhorov2013}.
In every case, the~required assertion follows from \cite[\S~III.3]{Matsuki}.
For instance, if $X$ is a member of Family \textnumero 3.14, then
there exists a blow up $\pi\colon X\to\mathbb{P}^3$ of a disjoint union of a line $L$ and a smooth quartic elliptic curve $C$,
and we have the~following commutative diagram:
\begin{equation}
\label{equation:3-6}
\xymatrix@R=1em{
&&\mathbb{P}^1\times\mathbb{P}^1\ar@/^1pc/@{->}[drr]^{\mathrm{pr}_2}\ar@/_1pc/@{->}[dll]_{\mathrm{pr}_1}&&\\%
\mathbb{P}^1&&&&\mathbb{P}^1\\%
&&X\ar@{->}[dll]^{\varpi_C}\ar@{->}[drr]^{\varpi_L}\ar@{->}[dd]_{\pi}\ar@{->}[uu]_{\eta}\ar@{->}[urr]^{\varphi_C}\ar@{->}[ull]_{\varphi_L}\\%
V_L\ar@{->}[uu]^{\phi_L}\ar@{->}[drr]_{\pi_L}&&&&V_C\ar@{->}[dll]^{\pi_C}\ar@{->}[uu]^{\phi_C}\\
&&\mathbb{P}^3&&}
\end{equation}
where $\pi_L$ and $\pi_C$ are blowups of curves $L$ and $C$, respectively,
$\varpi_L$ and $\varpi_C$ are blowups of the~strict transforms of the~curves $L$ and $C$, respectively,
$\phi_L$ is a $\mathbb{P}^2$-bundle, $\phi_C$ is a fibration into quadric surfaces,
$\varphi_L$ is a fibration into quintic del Pezzo surfaces,
$\varphi_C$ is a fibration into sextic del Pezzo surfaces,
the map $\eta$ is a (standard) conic bundle, $\mathrm{pr}_1$ and $\mathrm{pr}_2$ are projections to the~first and the~second factors, respectively.
Hence, we see that the~Mori cone $\overline{\mathrm{NE}(X)}$ is generated by extremal rays contracted by $\varpi_L$, $\varpi_C$ and $\eta$,
and the~group $G$ cannot non-trivially permute these extremal rays.
Similarly, one can verify the~required assertion in the~remaining cases using description of the~Mori cone $\overline{\mathrm{NE}}(X)$ given in \cite{Matsuki}.
\end{proof}

We now turn our attention to some certain stability threshold type invariants that allow estimations that can be used to prove K-polystability.
Let $X$ be a smooth Fano 3-fold, and let $p$ be a point in $X$. Recall that
$$
\delta_p(X)=\inf_{\substack{E/X\\ p\in C_X(E)}}\frac{A_{X}(E)}{S_{X}(E)},
$$
where $S_X(E)=\frac{1}{(-K_X)^3}\int\limits_0^\infty\mathrm{vol}(-K_X-uE)du$, and the infimum is taken over all prime divisors $E$ over $X$ whose center on $X$ contains $p$. Let $S$ be an irreducible smooth surface in $X$ that contains $p$.
Set
$$
\tau=\mathrm{sup}\Big\{u\in\mathbb{R}_{\geqslant 0}\ \big\vert\ \text{the divisor  $-K_X-uS$ is pseudo-effective}\Big\}.
$$
For~$u\in[0,\tau]$, let $P(u)$ be the~positive part of the~Zariski decomposition of the~divisor $-K_X-uS$,
and let $N(u)$ be its negative part. Then $S_X(S)=\frac{1}{-K_X^3}\int\limits_{0}^{\tau}\big(P(u)\big)^3du$.
For every prime divisor $F$ over $S$, set
$$
S\big(W^S_{\bullet,\bullet};F\big)=\frac{3}{(-K_X)^3}\int\limits_0^\tau\big(P(u)\big\vert_S\big)^2\cdot\mathrm{ord}_{F}\big(N(u)\big\vert_{S}\big)du+\frac{3}{(-K_X)^3}\int\limits_0^\tau\int\limits_0^\infty \mathrm{vol}\big(P(u)\big\vert_{S}-vF\big)dvdu.
$$

\begin{theorem}[{\cite{AbbanZhuang,Book}}]
\label{theorem:Hamid-Ziquan-Kento}
For any point $p\in S$ we have
$$
\delta_p(X)\geqslant\min\Bigg\{\frac{1}{S_X(S)},\inf_{\substack{F/S\\ p\in C_S(F)}}\frac{A_S(F)}{S\big(W^S_{\bullet,\bullet};F\big)}\Bigg\},
$$
where the~infimum is taken by all prime divisors over $S$ whose center on $S$ contains $p$.
\end{theorem}

If $E$ is a prime divisor over $X$ such that its center on $X$ is a curve in $S$, we have a similar result:

\begin{theorem}[{\cite{AbbanZhuang,Book}}]
\label{theorem:Hamid-Ziquan-Kento-1}
Let $E$ be a prime divisor over $X$ such that $C_X(E)$ is a curve $C\subset S$. Then
$$
\frac{A_X(E)}{S_X(E)}\geqslant\min\Bigg\{\frac{1}{S_X(S)},\frac{1}{S\big(W^S_{\bullet,\bullet};C\big)}\Bigg\}.
$$
\end{theorem}

\section{Smooth Fano 3-folds satisfying Condition~$\mathbf{(A)}$}
\label{Section:k-point}

\setlength{\epigraphwidth}{0.8\textwidth}
\epigraph{Common sense, do what it will, cannot avoid being surprised occasionally. The object of science is to spare it this emotion and create mental habits which shall be in such close accord with the~habits of the~world as to secure that nothing shall be unexpected.}{Bertrand Russel, quoted in \textit{The World of Null-A}, A.\ E.\ van Vogt, 1948}

In this section, we prove that every member in several deformation families of smooth Fano 3-folds satisfy Condition~$\mathbf{(A)}$.
As stated in the~introduction, there are 25 families of smooth Fano 3-folds containing K-polystable members that either have a non-K-polystable member or the~K-stability picture for all smooth elements is lacking.
Among them, we single out 10 families and prove in Section~\ref{subsection:K-poly-with-points} that any members in those families satisfies Condition~$\mathbf{(A)}$.
There are also 27 families of smooth Fano 3-folds where every smooth member is known to be non-K-polystable.
In Section~\ref{subsection:K-unstable-with-points} we show that every smooth member in 19 families (out of 27) always satisfies Condition~$\mathbf{(A)}$.
This leaves 8 families that contain exceptional cases in Section~\ref{subsection:Main-Theorem}.

\subsection{Families containing K-polystable members}
\label{subsection:K-poly-with-points}

\begin{lemma}
\label{lemma:1-10}
Let $X$ be a smooth Fano 3-fold in Family \textnumero 1.10.
Then $X$ satisfies Condition~$\mathbf{(A)}$.
\end{lemma}

\begin{proof}
Let $A$ be a finite abelian group in $\mathrm{Aut}(X)$. By Remark~\ref{remark:cyclic-Lefschetz}, we may assume that $A$ is not cyclic.
Recall that $|-K_X|$ is very ample and provides an $A$-equivariant embedding $X\hookrightarrow\mathbb{P}^{13}$
such that the action of $A$ on $\mathbb{P}^{13}$ is given by its faithful $14$-dimensional representation.

Suppose that $X$ contains an $A$-invariant smooth conic $C$.
Let $\phi\colon\widetilde{X}\to X$ be the~blowup along $C$, and let $E$ be the~$\phi$-exceptional surface.
Then it follows from \cite[Theorem~4.4.11]{IsPr99} and \cite[Corollary~4.4.3]{IsPr99}
that the~linear system $|-K_{\widetilde{X}}|$ is base point free, and the~linear system $|-K_{\widetilde{X}}-E|$ gives an $A$-equivariant birational map
$\chi\colon X\dasharrow Q$, where $Q$ is a smooth quadric 3-fold in $\mathbb{P}^{4}$.
Moreover, we have the~following $A$-equivariant commutative diagram
$$
\xymatrix@R=1em{
&\widehat{Q}\ar@{->}[ldd]_{\pi}\ar@{->}[rd]_{\beta}&& \widetilde{X}\ar@{->}[ld]^{\alpha}\ar@{->}[rdd]^{\phi}\ar@{-->}[ll]_{\zeta}&\\%
&&Y&&\\
Q&&&&X\ar@{-->}[llll]_{\chi}}
$$
where $\alpha$ is a small birational morphism given by $|-K_{\widetilde{X}}|$,
$Y$ is a Fano 3-fold with Gorenstein non-$\mathbb{Q}$-factorial terminal singularities with $-K_{Y}^3=16$,
the map $\zeta$ is a pseudo-isomorphism that flops the~curves contracted by $\alpha$,
$\pi$ is the~blowup of a smooth rational sextic curve $\Gamma\subset Q$,
and $\beta$ is a small birational morphism given by the~linear systems $|-K_{\widetilde{Q}}|$.
Observe that $A$ acts faithfully on $\Gamma$. Since $\Gamma\simeq\mathbb{P}^1$, this gives $A\simeq(\mathbb{Z}/2\mathbb{Z})^2$,
so $A$ fixes a point in $Q$ by Lemma~\ref{lemma:Klein-quadric},
and $A$ fixes a point in $X$ by Theorem~\ref{theorem:Kollar-Szabo}.

Thus, we see that $A$ fixes a point if $X$ contains an $A$-invariant smooth conic.
On the~other hand, we know from \cite{KuznetsovProkhorovShramov} that $A$ acts faithfully on the~Hilbert scheme of conics in $X$,
and this Hilbert scheme is isomorphic to $\mathbb{P}^2$.
Moreover, it follows from \cite{Kollar-Olaf,Schreyer} that singular conics on $X$ are parameterized by a quartic curve,
so applying Lemma~\ref{lemma:lift} again, we see that $X$ contains a (possibly singular) $A$-invariant conic $C$.
If $C$ is smooth, we are done. If $C$ is singular and reduced, then its singular point is fixed by $A$.
Similarly, if $C=2L$ for a line $L\subset X$, then $A$ fixes a point in $L$ by Corollary~\ref{corollary:odd-degree-curves}.
\end{proof}

\begin{lemma}
\label{lemma:2-9}
Let $X$ be a smooth Fano 3-fold in Family \textnumero 2.9.
Then $X$ satisfies Condition~$\mathbf{(A)}$.
\end{lemma}

\begin{proof}
Let $A$ be a finite abelian subgroup in $\mathrm{Aut}(X)$. By Lemma~\ref{lemma:Prokhorov}, we have an $A$-equivariant commutative diagram
$$
\xymatrix@R=1em{
&X\ar@{->}[ld]_{\pi}\ar@{->}[rd]^{\eta}&\\%
\mathbb{P}^3\ar@{-->}[rr]&&\mathbb{P}^2}
$$
in which $\pi$ is the~blowup of a smooth curve $C$ of degree $7$ and genus $5$, $\eta$ is a conic bundle with discriminant curve $\Delta\subset\mathbb{P}^2$ of degree $5$,
and the~dashed arrow is given by the~two-dimensional linear system of all cubic surfaces containing $C$.
Since $\Delta$ is $A$-invariant, it follows from Lemma \ref{lemma:lift} that $\mathbb{P}^2$ contains an $A$-invariant line $L$, so $\pi_*(\eta^*(L))$ is an $A$-invariant cubic surface.
Now, applying Lemma \ref{lemma:lift} again, we conclude that there is an $A$-fixed point on $\mathbb{P}^3$, hence a fixed point on $X$ by Theorem~\ref{theorem:Kollar-Szabo}.
\end{proof}

\begin{lemma}
\label{lemma:2-11}
Let $X$ be a smooth Fano 3-fold in Family \textnumero 2.11.
Then $X$ satisfies Condition~$\mathbf{(A)}$.
\end{lemma}

\begin{proof}
Let $A$ be a finite abelian subgroup in $\mathrm{Aut}(X)$.
By Lemma~\ref{lemma:Prokhorov}, there exists an $A$-equivariant birational morphism $\pi\colon X\to Y$ such that $Y$ is a smooth cubic 3-fold in $\mathbb{P}^4$,
and $\pi$ is a~blowup of an $A$-invariant line $L\subset Y$.
Applying Lemma~\ref{lemma:lift} and its Corollary~\ref{corollary:odd-degree-curves}, we see that $A$ fixes a point in $L$,
so it also fixes a point in $X$ by Theorem~\ref{theorem:Kollar-Szabo}.
\end{proof}

\begin{lemma}
\label{lemma:pointless-2-13}
Let $X$ be a smooth Fano 3-fold in Family \textnumero 2.13.
Then $X$ satisfies Condition~$\mathbf{(A)}$.
\end{lemma}

\begin{proof}
It follows from \cite{CheltsovPrzyjalkowskiShramov} that the~group $\mathrm{Aut}(X)$ is finite,
and it follows from Lemma~\ref{lemma:Prokhorov} that there exists an $\mathrm{Aut}(X)$-equivariant commutative diagram
$$
\xymatrix@R=1em{
&X\ar@{->}[ld]_{\pi}\ar@{->}[rd]^{\eta}&&\\%
Q\ar@{-->}[rr]&&\mathbb{P}^2}
$$
where $Q$ is a smooth quadric 3-fold in $\mathbb{P}^4$,
$\pi$ is the~blowup of a degree six smooth curve $C\subset Q$ with genus two,
$\eta$ is a conic bundle, and the~dashed arrow is given by the~linear subsystem in $|\mathcal{O}_{Q}(2)|$ consisting of all surfaces that contain $C$.
Set $L=\mathcal{O}_{Q}(1)\vert_{C}$. Then we have natural group homomorphisms
$$
\mathrm{Aut}(X)\simeq\mathrm{Aut}(Q,C)\hookrightarrow\mathrm{Aut}(\mathbb{P}^4,C)\simeq\mathrm{Aut}(C,[L])\hookrightarrow\mathrm{Aut}(C),
$$
where $\mathrm{Aut}(C,[L])$ is the~subgroup in $\mathrm{Aut}(C)$ consisting of all automorphisms of the~curve $C$ that preserve the~class of the~line bundle $L$ in $\mathrm{Pic}(C)$.
Therefore, in the~following we will identify
$\mathrm{Aut}(X)=\mathrm{Aut}(Q,C)$ and $\mathrm{Aut}(\mathbb{P}^4,C)=\mathrm{Aut}(C,[L])$,
and we consider $\mathrm{Aut}(X)$ as a subgroup of $\mathrm{Aut}(C)$.
Note that all possibilities for $\mathrm{Aut}(C)$ can be found using the~online database \cite{LGBT}.

Let $A$ be an abelian subgroup in $\mathrm{Aut}(Q,C)$. By Theorem~\ref{theorem:Kollar-Szabo}, it is enough to show that $A$ fixes a point in~$Q$.
Suppose that this is not the~case. Let us seek for a contradiction. First, $A$ is not cyclic by~Remark~\ref{remark:cyclic-Lefschetz}.
Second, it follows from Lemma~\ref{lemma:Klein-quadric} that $A\not\simeq (\mathbb{Z}/2\mathbb{Z})^2$.
Third, going through all possibilities for $\mathrm{Aut}(C)$ given by \cite{LGBT}, we see that $\mathrm{Aut}(C)\simeq(\mathbb{Z}/2\mathbb{Z})\rtimes\mathfrak{D}_4$,
$A\simeq (\mathbb{Z}/2\mathbb{Z})\times (\mathbb{Z}/6\mathbb{Z})$, the~signature of the~$A$-action on $C$ is $[0;2,2,6]$, and $C$ is isomorphic to the~following hyperelliptic curve:
\begin{equation}
\label{equation:hyperelliptic}
\{z^2=x^6-y^6\}\subset\mathbb{P}(1_{x},1_{y},3_{z}).
\end{equation}
Moreover, it follows from Lemma~\ref{lemma:lift} that $L$ is $A$-linearizable,
and it follows from \cite{Dolgachev1999} that the~only $A$-linearizable line bundle in $\mathrm{Pic}(C)$ of degree $6$ is $3K_Z$.
This implies that $L\sim 3K_Z$, and $C$ is projectively equivalent to the~image of the~curve \eqref{equation:hyperelliptic}
embedded into $\mathbb{P}^4$ via the~Veronese embedding $\mathbb{P}(1_{x},1_{y},3_{z})\hookrightarrow\mathbb{P}^4$ given by $[x:y:z]\mapsto[x^3:x^2y:xy^2:y^3:z]$.
Moreover, since $A$ is the~unique subgroup in $\mathrm{Aut}(C)$ isomorphic to $(\mathbb{Z}/2\mathbb{Z})\times (\mathbb{Z}/6\mathbb{Z})$,
we may assume that it acts on \eqref{equation:hyperelliptic} as follows:
\begin{align*}
[x:y:z]&\mapsto [x:y:-z],\\
[x:y:z]&\mapsto [\zeta_6x:y:z],
\end{align*}
where $\zeta_6$ is a primitive sixth root of unity.
Hence, choosing appropriate homogeneous coordinates $x_1$, $x_2$, $x_3$, $x_4$, $x_5$ on $\mathbb{P}^4$, we may assume that
$C=\{x_5^2=x_1^2-x_4^2, x_1x_4=x_2x_3, x_2^2=x_1x_3, x_3^2=x_2x_4\}\subset\mathbb{P}^4$,
and the~action of the~group $A$ on $\mathbb{P}^4$ is given by
\begin{align*}
[x_1:x_2:x_3:x_4:x_5]&\mapsto [x_1:x_2:x_3:x_4:-x_5],\\
[x_1:x_2:x_3:x_4:x_5]&\mapsto [\zeta_6^3x_1:\zeta_6^2x_2:\zeta_6x_3:x_4:x_5].
\end{align*}
In particular, the~(hyperelliptic) involution $[x_1:x_2:x_3:x_4:x_5]\mapsto [x_1:x_2:x_3:x_4:-x_5]$ pointwise fixes a hyperplane in $\mathbb{P}^4$,
so $A$ fixes a point in $Q$, which is a contradiction. In fact, one can check that there exists no smooth $A$-invariant quadric in $\mathbb{P}^4$ that contains the~curve $C$, which shows that there exists no smooth Fano 3-fold in Family \textnumero 2.13 that is acted faithfully by
the group $(\mathbb{Z}/2\mathbb{Z})\times (\mathbb{Z}/6\mathbb{Z})$.
\end{proof}

\begin{lemma}
\label{lemma:2-14-2-20-2-22}
Let $X$ be a smooth Fano 3-fold in one of Families \textnumero 2.14, \textnumero 2.20 or \textnumero 2.22.
Then $X$ satisfies Condition~$\mathbf{(A)}$.
\end{lemma}

\begin{proof}
In each case, $X$ can be obtained by blowing up the~smooth quintic del Pezzo 3-fold $V_5$ described in Lemma~\ref{lemma:V5} along a smooth curve.
Now, applying Lemmas~\ref{lemma:Prokhorov} and \ref{lemma:V5} and Theorem~\ref{theorem:Kollar-Szabo}, we obtain the~required assertion.
\end{proof}

\begin{lemma}
\label{lemma:2-17}
Let $X$ be a smooth Fano 3-fold in Family \textnumero 2.17.
Then $X$ satisfies Condition~$\mathbf{(A)}$.
\end{lemma}

\begin{proof}
Let $A$ be a finite abelian subgroup in $\mathrm{Aut}(X)$.
Then, by Lemma~\ref{lemma:Prokhorov}, we have the~following $A$-equivariant Sarkisov link:
$$
\xymatrix@R=1em{
&X\ar@{->}[ld]_{f}\ar@{->}[rd]^{\pi}&\\%
\mathbb{P}^3\ar@{-->}[rr]&&Q}
$$
where $Q$ is a smooth quadric 3-fold in $\mathbb{P}^4$,
$\pi$ is the~blowup of a smooth elliptic curve $C$ of degree $5$,
$f$ is a blowup of a smooth elliptic curve $Z$ of degree $5$,
and the~dashed arrow is given by the~linear system of cubic surfaces containing $Z$.
Now, applying Lemma~\ref{lemma:lift} to $\mathbb{P}^4$ and $Q$, we see that there exists an $A$-invariant hyperplane section of $Q$.
Taking its strict transform on $\mathbb{P}^3$, we obtain an $A$-invariant cubic surface in $\mathbb{P}^3$,
so it follows from Lemma~\ref{lemma:lift} that $A$ fixes a point in $\mathbb{P}^3$, which implies that $A$ fixes a point in $X$ by Theorem~\ref{theorem:Kollar-Szabo}.
\end{proof}

\begin{lemma}
\label{lemma:3-8}
Let $X$ be a smooth Fano 3-fold in Family \textnumero 3.8.
Then $X$ satisfies Condition~$\mathbf{(A)}$.
\end{lemma}

\begin{proof}
The 3-fold $X$ can be obtained as follows.
Consider a smooth divisor $Y\subset\mathbb{P}^2_{x,y,z}\times\mathbb{P}^2_{u,v,z}$ given by
$$
xf_2(u,v,w)+yg_2(u,v,w)+zh_2(u,v,w)=0,
$$
where $f_2$, $g_2$, $h_2$ are quadratic polynomials in $u$, $v$, $w$,
and a birational morphism $\pi\colon X\to Y$ that blows up a smooth fiber $C$ of the~projection $Y\to\mathbb{P}^2_{x,y,z}$.
Let $A$ be a finite abelian subgroup in $\mathrm{Aut}(X)$.
Then it follows from Lemma~\ref{lemma:Prokhorov} that $\pi$ is $A$-invariant.

Let $\mathrm{pr}_1\colon Y\to\mathbb{P}^2_{x,y,z}$ and $\mathrm{pr}_2\colon Y\to\mathbb{P}^2_{u,v,w}$ be the~projections to the~first and the~second factors, respectively.
Then both $\mathrm{pr}_1$ and $\mathrm{pr}_2$ are $A$-equivariant by Lemma~\ref{lemma:Prokhorov},
and it follows from Lemma~\ref{lemma:lift} that $A$ fixes a point $p\in\mathbb{P}^2_{u,v,w}$, since $\mathrm{pr}_2(C)$ is an $A$-invariant conic.
Let $F$ be the~fiber of $\mathrm{pr}_2$ over $p$. Then $F$ is $A$-invariant, and $\mathrm{pr}_1(F)$ is an $A$-invariant line, so $A$ fixes a point in $\mathrm{pr}_1(F)$ by Lemma~\ref{lemma:lift}.
Thus, since $\mathrm{pr}_1$ induces an isomorphism $F\simeq\mathrm{pr}_1(F)$, we see that $A$ fixes a point in $F$.
\end{proof}

\begin{lemma}
\label{lemma:3-11}
Let $X$ be a smooth Fano 3-fold in Family \textnumero 3.11.
Then $X$ satisfies Condition~$\mathbf{(A)}$.
\end{lemma}

\begin{proof}
Let $p$ be a point in $\mathbb{P}^3$,
and let $\phi\colon V_7\to\mathbb{P}^3$ be the~blowup of this point.
Then there exists a birational morphism $\pi\colon X\to V_7$ that blows up a smooth curve $C\subset V_7$ with $\phi(C)$ a smooth quartic elliptic curve passing through the~point $p$.
Now, let $A$ be a finite abelian subgroup in $\mathrm{Aut}(X)$.
Then it follows from Lemma~\ref{lemma:Prokhorov} that $\pi$ is $A$-equivariant,
and $A$ fixes a point in $V_7$ by Corollary~\ref{corollary:Pn-blow-up}, so $A$ also fixes a point in $X$ by Theorem~\ref{theorem:Kollar-Szabo}.
\end{proof}

\subsection{Families not containing K-polystable members}
\label{subsection:K-unstable-with-points}

\begin{lemma}
\label{lemma:2-26}
Let $X$ be a smooth Fano 3-fold in Family \textnumero 2.26.
Then $X$ satisfies Condition~$\mathbf{(A)}$.
\end{lemma}

\begin{proof}
In this case, $X$ can be obtained by blowing up del Pezzo 3-fold $V_5$ described in Lemma~\ref{lemma:V5} along a line.
By applying Lemmas~\ref{lemma:Prokhorov} and \ref{lemma:V5} and Theorem~\ref{theorem:Kollar-Szabo}, we obtain the~required assertion.
\end{proof}

\begin{lemma}
\label{lemma:2-28-2-30}
Let $X$ be a smooth Fano 3-fold in Families \textnumero 2.28 or \textnumero 2.30.
Then $X$ satisfies Condition~$\mathbf{(A)}$.
\end{lemma}

\begin{proof}
The 3-fold $X$ can be obtained by blowing up $\mathbb{P}^3$ along a smooth plane curve $C$ of degree $3$ or~$2$.
Let $A$ be a finite abelian subgroup in $\mathrm{Aut}(X)$.
Then, by Lemma~\ref{lemma:Prokhorov}, this blowup is $A$-equivariant.
Moreover, the~group $A$ leaves invariant the~plane in $\mathbb{P}^3$ that contains $C$, so $A$ fixes a point in $\mathbb{P}^3$ by Lemma~\ref{lemma:lift},
which implies that $A$ fixes a point in $X$ by Theorem~\ref{theorem:Kollar-Szabo}.
\end{proof}

\begin{lemma}
\label{lemma:2-31}
Let $X$ be a smooth Fano 3-fold in Family \textnumero 2.31.
Then $X$ satisfies Condition~$\mathbf{(A)}$.
\end{lemma}

\begin{proof}
Let $A$ be a finite abelian subgroup in $\mathrm{Aut}(X)$.
Then, it follows from Lemma~\ref{lemma:Prokhorov} that there exists an $A$-equivariant birational morphism
$\pi\colon X\to Q$ such that $Q$ is a smooth quadric 3-fold in $\mathbb{P}^4$, and $\pi$ is a blowup of a line in $Q$.
Let $E$ be the~$\pi$-exceptional surface, then $E$ is $A$-invariant and $E\simeq\mathbb{F}_1$.
Now, applying Corollary~\ref{corollary:Pn-blow-up}, we see that $A$ fixes a point in $E$,
so it fixes a point in $X$ by Theorem~\ref{theorem:Kollar-Szabo}.
\end{proof}

\begin{lemma}
\label{lemma:2-35}
Let $X$ be a smooth Fano 3-fold in Family \textnumero 2.35.
Then $X$ satisfies Condition~$\mathbf{(A)}$.
\end{lemma}

\begin{proof}
The 3-fold $X$ is unique and is often called $V_7$ --- it can be obtained by blowing up $\mathbb{P}^3$ at a point,
so required assertion is a special case of Corollary~\ref{corollary:Pn-blow-up}.
\end{proof}

\begin{lemma}
\label{lemma:2-36}
Let $X$ be a smooth Fano 3-fold in Family \textnumero 2.36.
Then $X$ satisfies Condition~$\mathbf{(A)}$.
\end{lemma}

\begin{proof}
In this case, $X\simeq\mathbb{P}( {\mathcal O}_{\mathbb{P}^2} \oplus {\mathcal O}_{\mathbb{P}^2}(-2))$,
and it possesses two extremal contractions: a divisorial contraction $\pi\colon X \to\mathbb{P}(1,1,1,2)$
and a $\mathbb{P}^1$-bundle $\eta\colon X \to \mathbb{P}^2$.
Let $A$ be a finite abelian subgroup in $\mathrm{Aut}(X)$.
Then, by Lemma~\ref{lemma:Prokhorov}, both $\pi$ and $\eta$ are $A$-equivariant.
Let $E$ be the~$\pi$-exceptional divisor. Then $E\simeq\mathbb{P}^2$, and this isomorphism gives $\mathcal{O}_E(E|_E)\simeq \mathcal{O}_{\mathbb{P}^2}(-2)$.
Moreover, since $\mathcal{O}_E(E\vert_{E})$ is isomorphic to the~normal bundle of the~surface $E$ in $X$,
we see that this line bundle is $A$-linearizable.
Then the~line bundle $\mathcal{O}_E(-K_E-E\vert_{E})\simeq \mathcal{O}_{\mathbb{P}^2}(1)$ is also $A$-linearizable,
which implies that the~$A$-action on $E\simeq\mathbb{P}^2$ is given by a $3$-dimensional representation of the~group $A$,
so $A$ fixes a point in $E$.
\end{proof}

\begin{lemma}
\label{lemma:3-14-4-9}
Let $X$ be a smooth Fano 3-fold in one of Families \textnumero 3.14, \textnumero 3.16, \textnumero 3.23, \textnumero 3.26,
\textnumero 3.29, \textnumero 3.30 or \textnumero 4.9. Then $X$ satisfies Condition~$\mathbf{(A)}$.
\end{lemma}

\begin{proof}
Let $p$ be a point in $\mathbb{P}^3$,
and let $\phi\colon V_7\to\mathbb{P}^3$ be the~blowup of this point.
Then there exists a birational morphism $\pi\colon X\to V_7$ that blows up a smooth curve $C$.
The curve $C$ irreducible if $X$ is contained in Families \textnumero 3.14, \textnumero 3.16, \textnumero 3.23, \textnumero 3.26,
\textnumero 3.29 or \textnumero 3.30.
On the~other hand, if $X$ is contained in Family \textnumero 4.9, then $C$ is the~strict transform of two skew lines in $\mathbb{P}^3$ such that one of them contains $p$.
In every case, it follows from Lemma~\ref{lemma:Prokhorov} that $\pi$ is $A$-equivariant for any finite abelian subgroup $A\subset\mathrm{Aut}(X)$, so $A$ fixes a point in $X$ by Theorem~\ref{theorem:Kollar-Szabo},
because $A$ fixes a point in $V_7$ by Corollary~\ref{corollary:Pn-blow-up}.
\end{proof}

\begin{lemma}
\label{lemma:3-18}
Let $X$ be a smooth Fano 3-fold in Family \textnumero 3.18.
Then $X$ satisfies Condition~$\mathbf{(A)}$.
\end{lemma}

\begin{proof}
The Fano 3-fold $X$ can be obtained as a blowup
$\pi\colon X\to \mathbb{P}^3$ along a disjoint union of a smooth conic $C$ and a line $L$.
Let $A$ be a finite abelian subgroup in $\mathrm{Aut}(X)$.
Then, by Lemma~\ref{lemma:Prokhorov}, $\pi$ is $A$-equivariant.
In particular, the~plane containing $C$ is $A$-invariant, so, by Lemma~\ref{lemma:lift}, $A$ fixes a point in $\mathbb{P}^3$, hence $A$ also fixes a point in $X$ by Theorem~\ref{theorem:Kollar-Szabo}.
\end{proof}

\begin{lemma}
\label{lemma:3-21}
Let $X$ be a smooth Fano 3-fold in Family \textnumero 3.21.
Then $X$ satisfies Condition~$\mathbf{(A)}$.
\end{lemma}

\begin{proof}
In this case, there exists a blowup $\pi\colon X\to\mathbb{P}^1\times\mathbb{P}^2$ of a smooth rational curve $C$ of degree $(2,1)$.
Let $\mathrm{pr}_1\colon \mathbb{P}^1\times\mathbb{P}^2\to \mathbb{P}^1$ and $\mathrm{pr}_2\colon \mathbb{P}^1\times\mathbb{P}^2\to \mathbb{P}^2$ be projections to the~first and second factors, respectively. Then $\mathrm{pr}_1\vert_{C}\colon C\to\mathbb{P}^1$ is a double cover, and $\mathrm{pr}_2(C)$ is a line in $\mathbb{P}^2$.

Let $A$ be a finite abelian subgroup in $\mathrm{Aut}(X)$.
Then it follows from Lemma~\ref{lemma:Prokhorov}, $\pi$ is $A$-equivariant.
Since $A$ leaves the~line $\mathrm{pr}_2(C)$ invariant, it follows from Lemma~\ref{lemma:lift} that $A$ fixes a point $p\in\mathrm{pr}_2(C)$.
Since the~fiber of the~projection $\mathrm{pr}_2$ over the~point $p$ intersects the~curve $C$ by one point, we see that this point is fixed by $A$,
so $A$ also fixes a point in $X$ by Theorem~\ref{theorem:Kollar-Szabo}.
\end{proof}

\begin{lemma}
\label{lemma:3-22}
Let $X$ be a smooth Fano 3-fold in Family \textnumero 3.22.
Then $X$ satisfies Condition~$\mathbf{(A)}$.
\end{lemma}

\begin{proof}
Let $\mathrm{pr}_1\colon \mathbb{P}^1\times\mathbb{P}^2\to\mathbb{P}^1$ be the~projection to the~first factor,
let $H_1$ be its fiber, and let $C$ be a~smooth conic in $H_1\cong\mathbb{P}^2$.
Then there is a blowup $\pi\colon X\to \mathbb{P}^1\times\mathbb{P}^2$ along the~curve $C$. Let $A$ be a finite abelian subgroup in $\mathrm{Aut}(X)$.
Then it follows from Lemma~\ref{lemma:Prokhorov} that both $\pi$ and $\mathrm{pr}_1$ are $A$-equivariant,
and both $H_1$ and $C$ are $A$-invariant. Hence, it follows from Lemma~\ref{lemma:lift} that $A$ fixes a point in $H_1$,
so $A$ also fixes a point in $X$ by Theorem~\ref{theorem:Kollar-Szabo}.
\end{proof}

\begin{lemma}
\label{lemma:3-24}
Let $X$ be a smooth Fano 3-fold in Family \textnumero 3.24.
Then $X$ satisfies Condition~$\mathbf{(A)}$.
\end{lemma}

\begin{proof}
Let $A$ be a finite abelian subgroup in $\mathrm{Aut}(X)$.
Then it follows from Lemma~\ref{lemma:Prokhorov} that there is an $A$-equivariant blowup $\pi\colon X\to\mathbb{P}^1\times\mathbb{P}^2$ of a smooth $A$-invariant curve $C$ of degree~$(1,1)$.

Let $\mathrm{pr}_2\colon \mathbb{P}^1\times\mathbb{P}^2\to\mathbb{P}^2$ be the projection to the second factor.
Then $\mathrm{pr}_2$ is $A$-equivariant, and $\mathrm{pr}_2(C)$ is an $A$-invariant line in $\mathbb{P}^2$,
and it follows from Lemma~\ref{lemma:lift} that $A$ fixes a point in $\mathrm{pr}_2(C)$. Then $A$ fixes a point in $C$,
since $\mathrm{pr}_2|_C\colon C \to \mathrm{pr}_2(C)$ is an isomorphism, so $A$ fixes a point in $X$ by Theorem~\ref{theorem:Kollar-Szabo}.
\end{proof}

\begin{lemma}
\label{lemma:4-5}
Let $X$ be a smooth Fano 3-fold in Family \textnumero 4.5.
Then $X$ satisfies Condition~$\mathbf{(A)}$.
\end{lemma}

\begin{proof}
In this case, there exists a birational morphism $\pi\colon X\to \mathbb{P}^1\times\mathbb{P}^2$ that blows up two smooth disjoint curves $C$ and $L$ such that $C$ is a curve of degree $(2,1)$, and $L$ is a curve of degree $(1,0)$.
Let $A$ be a finite abelian subgroup in $\mathrm{Aut}(X)$.
Then it follows from Lemma~\ref{lemma:Prokhorov} that $\pi$ is $A$-equivariant,
and both curves $C$ and $L$ are $A$-invariant.
Thus, blowing up $\mathbb{P}^1\times\mathbb{P}^2$ along $C$ gives us $A$-equivariant birational morphism to the~smooth Fano 3-folds in Family \textnumero 3.21,
and $A$ fixes a point in this 3-fold by Lemma~\ref{lemma:3-21}, so $A$ also fixes a point in $X$ by Theorem~\ref{theorem:Kollar-Szabo}.
\end{proof}

\begin{lemma}
\label{lemma:4-11}
Let $X$ be a smooth Fano 3-fold in Family \textnumero 4.11.
Then $X$ satisfies Condition~$\mathbf{(A)}$.
\end{lemma}

\begin{proof}
Let $V=\mathbb{P}^1 \times\mathbb{F}_1$, let $S$ be a fiber of the~natural projection $V\to\mathbb{P}^1$,
and let $C$ be the~$(-1)$-curve in $S\cong\mathbb{F}_1$.
Then there exists a birational morphism $\pi\colon X \to V$ that is the~blowup of the~curve $C$. Let $A$ be a finite abelian subgroup in $\mathrm{Aut}(X)$.
Then it follows from Lemma~\ref{lemma:Prokhorov} that $\pi$ is $A$-equivariant,
and $S$ is $A$-invariant, so $A$ fixes a point in $S$ by Corollary~\ref{corollary:Pn-blow-up}. Hence $A$ fixes a point in $X$ also by Theorem~\ref{theorem:Kollar-Szabo}.
\end{proof}

\section{K-polystable Fano 3-folds not satisfying Condition~$\mathbf{(A)}$}
\label{section:pointless-3-folds}

\setlength{\epigraphwidth}{0.8\textwidth}
\epigraph{It was a pleasure to remember the crystal-like clarity of the man's writing, the careful semantic consideration given to every multi-ordinal word used, the breadth of intellect and understanding of the human body-and-mind-as-a-whole.}{\textit{The World of Null-A}, A.\ E.\ van Vogt, 1948}

In this section, we work through the~18 families of Fano 3-folds that contain K-polystable elements but K-polystability is not known or does not hold for all elements.
They also exhibit the~phenomenon that their smooth elements do not always satisfies Condition~$\mathbf{(A)}$.
We will constantly use the~following lemma.

\begin{lemma}
\label{lemma:F}
Let $X$ be a smooth Fano 3-fold, and $G$ be a finite subgroup in $\mathrm{Aut}(X)$.
Suppose that $X$ is not K-polystable.
Then there exists a $G$-invariant prime divisor $\mathbf{F}$ over $X$ such that
$$
1\geqslant\delta(X)=\frac{A_X(\mathbf{F})}{S_X(\mathbf{F})}.
$$
\end{lemma}

\begin{proof}
It follows from the~valuative criterion for K-stability \cite{Fujita2019,Li2017} that $\delta(X)\leqslant 1$.
Since $\delta(X)<\frac{4}{3}$, it follows from \cite[Theorem 1.2]{LXZ} that there exists a prime divisor $\mathbf{F}$ over $X$ such that
$$
\delta(X)=\frac{A_X(\mathbf{F})}{S_X(\mathbf{F})}.
$$
Moreover, if $\delta(X)<1$, then it follows from \cite[Theorem 4.4]{Zhuang2021} that we may choose $\mathbf{F}$ to be $G$-invariant.
Similarly, if $\delta(X)=1$, then we can also assume that $\mathbf{F}$ is $G$-invariant by \cite[Corollary 4.14]{Zhuang2021}, because $X$ is not K-polystable.
\end{proof}

In the arguments in this section, we will use Lemma~\ref{lemma:F} as follows.
In each case, we denote by $X$ the~smooth Fano 3-fold which does not satisfy Condition~$\mathbf{(A)}$,
and we denote by $A$ a finite abelian subgroup in $\mathrm{Aut}(X)$ that does not fix points on $X$.
We aim to prove that $X$ is K-polystable. We assume that $X$ is not K-polystable.
Then, by Lemma~\ref{lemma:F}, there exists an $A$-invariant prime divisor $\mathbf{F}$ over $X$ such that
$$
1\geqslant\delta(X)=\frac{A_X(\mathbf{F})}{S_X(\mathbf{F})}.
$$
Let $Z\subset X$ be the~center of the~divisor $\mathbf{F}$.
Then, in each case under consideration, $Z$ is not a surface by \mbox{\cite[Theorem~3.17]{Book}}.
On the~other hand, since $A$ does not fix points on $X$, we conclude that $Z$ is an irreducible curve,
and $\delta_p(X)\leqslant 1$ for every point $p\in Z$.
Then, we will seek for a contradiction, often by showing that $\delta_p(X)>1$ for a sufficiently general point $p\in Z$.

\begin{lemma}
\label{lemma:pointless-1-9}
Let $X$ be a smooth Fano 3-fold in Family \textnumero 1.9, and let $A$ be a finite abelian subgroup in $\mathrm{Aut}(X)$ that does not fix points on $X$.
Then $X$ is K-polystable.
\end{lemma}

\begin{proof}
First, we observe that the~linear system $|-K_X|$ gives an $A$-equivariant embedding $X\hookrightarrow \mathbb{P}^{11}$ such that the action of the~group $A$ on $\mathbb{P}^{11}$
is induced by its faithful $12$-dimensional representation.

Suppose that $X$ is not is K-polystable. Let $\mathbf{F}$ be an $A$-invariant divisor whose existence is asserted by Lemma~\ref{lemma:F},
and let $Z$ be its center on $X$. Then, as we explained above, $Z$ is an irreducible curve.
Now, we choose a very general surface $S\in |-K_{X}|$.
Then $S$ is a smooth K3 surface with $\mathrm{Pic}(S)=\mathbb{Z}[-K_{X}|_S]$, so that it follows from \cite{Knutsen} and \cite[Theorem~A]{AbbanZhuangSeshadri} that $\delta(S,-K_{X}\vert_{S})\geqslant\frac{4}{5}$,
so that
$$
\delta(X)\leqslant 1 \leqslant\frac{4}{3}\delta(S,-K_{X}\vert_{S}).
$$
We are now in a position to apply \cite[Corollary 5.6]{AbbanZhuangSeshadri} to $\mathbf{F}$ over $X$.
Hence, it follows from \cite[Corollary 5.6]{AbbanZhuangSeshadri} that at least one of the~following two cases holds:
\begin{enumerate}
\item either there exists an effective $A$-invariant $\mathbb{Q}$-divisor $D$ on the~3-fold $X$ such that $D\sim_{\mathbb{Q}} -K_X$
and $Z$ is a center of non-log canonical singularities of the~log pair $(X,\frac{1}{2}D)$, so, in particular,
the log pair $(X,\frac{1}{2}D)$ is not log canonical along the~curve $Z$;

\item or there exists an $A$-invariant mobile linear system $\mathcal{M}\subset |-nK_{X}|$ such that
$Z$ is a center of non-klt singularities of the~log pair $(X,\frac{1}{2n}\mathcal{M})$.
\end{enumerate}
In the~second case, if $M_1$ and $M_2$ are general surfaces in $\mathcal{M}$, then it follows from \cite[Theorem 3.1]{Corti2000} that
$
M_1\cdot M_2=mZ+\Delta
$
for some positive integer $m\geqslant 16n^2$ and some effective one-cycle $\Delta$ on the~3-fold $X$,
which implies that
$$
18n^2=-K_X\cdot M_1\cdot M_2=m(-K_X)\cdot Z+(-K_X)\cdot\Delta\geqslant m(-K_X)\cdot Z\geqslant 16n^2(-K_X)\cdot Z,
$$
so that $-K_X\cdot Z=1$, which is impossible by Corollary~\ref{corollary:odd-degree-curves}.

Now, arguing as in the~proof of \cite[Theorem~1.52]{Book}, we can replace the~effective $\mathbb{Q}$-divisor $D$
with another $A$-invariant $\mathbb{Q}$-divisor $D^\prime$ on the~3-fold $X$ such that $D^\prime\sim_{\mathbb{Q}} -K_X$,
the log pair $(X,\lambda D^\prime)$ has log canonical singularities for some positive rational number $\lambda<\frac{1}{2}$
such that the~singularities of the~log pair $(X,\lambda D^\prime)$ are non-klt (not Kawamata log terminal),
and the~locus $\mathrm{Nklt}(X,\lambda D^\prime)$ is geometrically irreducible,
and consists of a minimal center of log canonical singularities of the~pair $(X,\lambda D^\prime)$.
Here, we implicitly used Nadel's vanishing theorem and Koll\'ar--Shokurov connectedness theorem, see \cite[Appendix~A.1]{Book}.

Set $C=\mathrm{Nklt}(X,\lambda D^\prime)$. Then $C$ is not a surface, since $\mathrm{Pic}(X)$ is generated by $-K_X$.
Similarly, as above, we see that $C$ is not a point, because  $A$ does not fix points in $X$.
Thus, we see that $C$ is an irreducible curve.
Then it follows from the~proof of  \cite[Theorem~1.52]{Book} that $C$ is a smooth rational curve with
$$
-K_X\cdot C\leqslant \frac{2}{1-\lambda}<4.
$$
But $-K_X\cdot C$ is even by Corollary~\ref{corollary:odd-degree-curves}, so we see that $-K_X\cdot C=2$.

Let $\phi\colon\widetilde{X}\to X$ be the~blowup of the~curve $C$, and let $E$ be the~$\phi$-exceptional surface.
Then it follows from \cite[Theorem~4.4.11]{IsPr99} and \cite[Corollary~4.4.3]{IsPr99}
that $|-K_{\widetilde{X}}|$ is base point free, and $|-K_{\widetilde{X}}-E|$ gives a birational map
$\chi\colon X\dasharrow \mathbb{P}^2$ such that we have the~following commutative diagram:
$$
\xymatrix@R=1em{
&\widetilde{X}\ar@{->}[rd]_{\alpha}\ar@{->}[ldd]_{\phi}\ar@{-->}[rr]^{\zeta}&&V\ar@{->}[rdd]^{\pi}\ar@{->}[ld]^{\beta} &\\%
&&Y&&\\
X\ar@{-->}[rrrr]^{\chi}&&&&\mathbb{P}^2}
$$
where $Y$ is a Fano 3-fold with Gorenstein non-$\mathbb{Q}$-factorial terminal singularities,
$\alpha$ and $\beta$ are small birational morphisms, $\zeta$ is a pseudo-isomorphism that flops the~curves contracted by $\alpha$,
and $\pi$ is a conic bundle.
This shows that the~cone of effective divisors of the~3-fold $\widetilde{X}$ is generated by $E$ and $-K_{\widetilde{X}}-E$.
On the other hand, it follows from \cite[Proposition~9.5.13]{Lazarsfeld} that
$$
\mathrm{mult}_{C}(D^\prime)\geqslant\frac{1}{\lambda}>2.
$$
Thus, if $\widetilde{D}^\prime$ is the~strict transform on $\widetilde{X}$ of the~divisor $D^\prime$, then
$\widetilde{D}^\prime\sim_{\mathbb{Q}}(-K_{\widetilde{X}}-E)-(\mathrm{mult}_{C}(D^\prime)-2)E$,
which is a contradiction, since $\mathrm{mult}_{C}(D^\prime)>2$.
\end{proof}

\begin{lemma}
\label{lemma:pointless-2-5-2-10-3-7}
Let $X$ be a smooth Fano 3-fold in Families \textnumero 2.5, \textnumero 2.10 or \textnumero 3.7,
and let $A$ be a finite abelian subgroup in $\mathrm{Aut}(X)$ that does not fix points on $X$.
Then $X$ is K-polystable.
\end{lemma}

\begin{proof}
In the~following diagram
\begin{equation}
\label{equation:2-5-link}
\xymatrix@R=1em{
&X\ar@{->}[ld]_{\pi}\ar@{->}[rd]^{\phi}&&\\%
V&&\mathbb{P}^1}
\end{equation}
$V$ is one of the~following 3-folds: a~smooth cubic 3-fold in $\mathbb{P}^4$ (if $X$ is contained in Family \textnumero 2.5),
a~smooth complete intersection of two quadrics in $\mathbb{P}^5$ (if $X$ is contained in Family \textnumero 2.10), or
a smooth divisor of degree $(1,1)$ in $\mathbb{P}^2\times\mathbb{P}^2$ (if $X$ is contained in Family \textnumero 3.7),
the morphism $\pi$ is the~blowup of an $A$-invariant smooth elliptic curve $C\subset V$,
and $\phi$ is a fibration into del Pezzo surfaces of degree $d$, where
$$
d=\left\{\aligned
&3\ \text{if $X$ contained in Family \textnumero 2.5}, \\
&4\ \text{if $X$ contained in Family \textnumero 2.10}, \\
&6\ \text{if $X$ contained in Family \textnumero 3.7}.
\endaligned
\right.
$$
Note that every fiber of $\phi$ is irreducible, reduced and normal.
Moreover, if $X$ is contained in Family \textnumero 2.10 or in Family \textnumero 3.7,
then all fibers of $\phi$ have at worst Du Val singularities.
Furthermore, if $X$ is contained in Family \textnumero 2.5 or Family~\textnumero 2.10,
then \eqref{equation:2-5-link} is $A$-equivariant by Lemma~\ref{lemma:Prokhorov}.
Finally, if $X$ is contained in Family~\textnumero 3.7, then \eqref{equation:2-5-link} is also $A$-equivariant,
which can be shown using the~description of the~Mori cone $\overline{\mathrm{NE}}(X)$ presented in \cite[\S~III.3]{Matsuki}.

Suppose that $X$ is not K-polystable. Let $\mathbf{F}$ be an $A$-invariant divisor whose existence is asserted by Lemma~\ref{lemma:F},
and let $Z$ be its center on $X$. Then, as explained above, $Z$ is an irreducible $A$-invariant curve.

Let $E$ be the~$\pi$-exceptional surface.
Then the~natural projection $E\to C$ and the~restriction $\phi\vert_{E}\colon E\to\mathbb{P}^1$ give isomorphism $E\simeq C\times\mathbb{P}^1$.
We claim that $Z\not\subset E$. Indeed, suppose that $Z\subset E$. Then $\pi(Z)=C$, since otherwise $\pi(Z)$ would be an $A$-fixed point which is impossible by Theorem~\ref{theorem:Kollar-Szabo}. Let $H$ be the~class of a divisor in $\mathrm{Pic}(V)$ such that $-K_{V}\sim 2H$,
and let $u$ be a~non-negative real number. Then
$$
-K_{X}-uE\sim_{\mathbb{R}}\pi^*(2H)-(1+u)E,
$$
which implies that the divisor $-K_{X}-uE$ is pseudoeffective if and only if  $u\leqslant 1$,
and for every $u\in[0,1]$, the~divisor $-K_{X}-uE$ is nef.
This gives $S_{X}(E)=\frac{1}{4d}\int\limits_{0}^{1}\big(-K_{X}-uE\big)^3du=\frac{1}{4d}\int\limits_{0}^{1}d(2u^3-6u+4)du=\frac{3}{8}$.
Then we have $S(W_{\bullet,\bullet}^{E}; Z)\geqslant 1$ by Theorem~\ref{theorem:Hamid-Ziquan-Kento-1}, where
$S\big(W_{\bullet,\bullet}^{E}; Z\big)=\frac{3}{4d}\int\limits_{0}^{1}\int\limits_0^\infty \mathrm{vol}\big((-K_{X}-uE)\vert_{S}-vZ\big)dvdu$.
Let $\mathbf{s}$ be a~fiber of the~restriction $\phi\vert_{E}\colon E\to\mathbb{P}^1$,
and let $\mathbf{f}$ be a~fiber of the~natural projection $E\to C$.
Then $Z\equiv a\mathbf{s}+b\mathbf{f}$
for some non-negative integers $a$ and $b$, where $a\geqslant 1$, since $\pi(Z)=C$.
This gives
\begin{multline*}
1\leqslant S\big(W_{\bullet,\bullet}^{E}; Z\big)\leqslant\frac{3}{4d}\int\limits_{0}^{1}\int\limits_0^\infty \mathrm{vol}\Big(\big(-K_{X}-uE\big)\big\vert_{S}-v\mathbf{s}\Big)dvdu=\\
=\frac{3}{4d}\int\limits_{0}^{1}\int\limits_{0}^{1+u}\big((1+u-v)\mathbf{s}+d(1-u)\mathbf{f}\big)^2dvdu=\frac{3}{4d} \int\limits_{0}^{1} \int\limits_{0}^{1+u} 2d(1-u)(1+u-v)dvdu=\frac{11}{16},
\end{multline*}
which is absurd. Hence, we conclude that $Z\not\subset E$.

Let $p$ be sufficiently general point in $Z$, and let $S$ be the~fiber of $\phi$ that contains $p$. Then $\delta_p(X)\leqslant 1$. Note that $p\not\in E$, since $Z\not\subset E$.
Thus, if $S$ has Du Val singularities, then
\begin{equation}
\label{equation:2-5}
1\geqslant\delta_p(X)\geqslant\mathrm{min}\Big\{\frac{16}{11},\frac{16}{15}\delta(S)\Big\}.
\end{equation}
Indeed, if $X$ is contained in Family \textnumero 2.5 or in Family \textnumero 2.10, this follows from \cite[Lemma 2.1]{CheltsovDenisovaFujita}.
Similarly, if $X$ is contained in Family \textnumero 3.7,
then \eqref{equation:2-5} follows from the~proof of \cite[Lemma 2.1]{CheltsovDenisovaFujita}.
Hence, if $S$ has Du Val singularities, then \eqref{equation:2-5} gives $\delta(S)<1$, so $S$ is K-unstable.
On the~other hand, if $\phi(Z)=\mathbb{P}^1$, then the~surface $S$ is smooth and it is known to be K-polystable \cite{Book}.
Hence, we conclude that $\phi(Z)$ is a point in $\mathbb{P}^1$, so $Z\subset S$, which implies that $S$ is $A$-invariant.
Moreover, if $S$ has Du Val singularities, then the~surface $S$ is K-polystable by Theorem~\ref{theorem:dP}, which is a contradiction.
Hence, we see that $X$ is contained in Family \textnumero 2.5, and $S$ is a cubic cone,
so its vertex is fixed by $A$, which is a contradiction.
\end{proof}

\begin{lemma}
\label{lemma:pointless-2-12}
Let $X$ be a smooth Fano 3-fold in Family \textnumero 2.12, and let $A$ be a finite abelian subgroup in $\mathrm{Aut}(X)$ that does not fix points on $X$.
Then $X$ is K-polystable.
\end{lemma}

\begin{proof}
The required assertion follows from \cite[Corollary~3]{CheltsovLiMauPinardin}.
\end{proof}

\begin{lemma}
\label{lemma:pointless-2-16}
Let $X$ be a smooth Fano 3-fold in Family \textnumero 2.16, and let $A$ be a finite abelian subgroup in $\mathrm{Aut}(X)$ that does not fix points on $X$.
Then $X$ is K-polystable.
\end{lemma}

\begin{proof}
Using Lemma~\ref{lemma:Prokhorov},
we see that there exists an $A$-equivariant birational morphism $f\colon X\to V$ such that $V$ is a smooth complete intersection of two quadrics in $\mathbb{P}^5$,
and $f$ is a blowup of a conic $C\subset V$.
Moreover, we have the~following $A$-equivariant commutative diagram:
$$
\xymatrix@R=1em{
&X\ar@{->}[ld]_{f}\ar@{->}[rd]^{\pi}&&\\%
V\ar@{-->}[rr]&&\mathbb{P}^2}
$$
where $\pi$ is a conic bundle whose discriminant curve is a (possibly singular) reduced quartic curve $\Delta_4\subset\mathbb{P}^2$,
and the~dashed arrow is the~projection from the~plane in $\mathbb{P}^5$ containing $C$.
Note that $A$ does not fix points in $V$ by Theorem~\ref{theorem:Kollar-Szabo}, but it follows from Lemma~\ref{lemma:lift} that $A$ fixes a point in $\mathbb{P}^2$.

We claim that $A$ does not fix points in the~curve $\Delta_4$. Indeed, suppose that $A$ fixes a point $p\in\Delta_4$.
Let $Z$ be the~fiber of the~conic bundle $\pi$ over $p$ with reduced structure. Then $Z$ is an $A$-invariant curve.
Moreover, if $p$ is a smooth point of the~curve $\Delta_4$, then $Z$ has one singular point, so this point must be fixed by $A$, which contradicts our assumption.
Thus, we see that $p$ is a singular point of the~curve $\Delta_4$.
Then $f(Z)$ is a line in $V$ that intersects  the conic $C$ by one point,
so the~intersection point $f(Z)\cap C$ must be fixed by $A$, which is impossible.
Thus, we see that $A$ does not fix points in $\Delta_4$.

Let $E$ be the~$f$-exceptional surface, and let $H$ be the~preimage on $X$ of a sufficiently general hyperplane section of the~complete intersection $V$.
Then $\pi$ is given by $|H-E|$.
We claim that $E\simeq\mathbb{P}^1\times\mathbb{P}^1$ or $E\simeq\mathbb{F}_2$.
Indeed, the~normal bundle of the~conic $C\simeq \mathbb{P}^1$ is $\mathcal{O}_{\mathbb{P}^1}(a)\oplus \mathcal{O}_{\mathbb{P}^1}(b)$, where $a\leqslant b$ and $a+b=2$.
Set $n=b-a$. Then $E\simeq\mathbb{F}_{n}$.  Let $\mathbf{l}$ be a fiber of the~natural projection $E\to C$, and let $\mathbf{s}$ be a section of this projection such that $\mathbf{s}^2=-n$.
Then $H\vert_{E}\sim 2\mathbf{l}$ and $-E\big\vert_{E}\sim\mathbf{s}+\frac{n-2}{2}\mathbf{l}$.
Thus, since $(H-E)\vert_{E}$ is nef, we get
$$
0\leqslant (H-E)\vert_{E}\cdot\mathbf{s}=\frac{2-n}{2},
$$
so $n=0$ or $n=2$ as claimed.
If $n=0$, the~restriction $\pi\vert_{E}\colon E\to\mathbb{P}^2$ is a double cover ramified in a conic.
Similarly, if $n=2$, then $\pi\vert_{E}$ is a contraction of the~$(-2)$-curve $\mathbf{s}$ followed by the~double cover ramified in a union of two distinct lines.

Now, we are ready to show that $X$ is K-polystable.
Assume the contrary and let $\mathbf{F}$ be an $A$-invariant divisor whose existence is asserted by Lemma~\ref{lemma:F},
and let $Z$ be its center on $X$. Then, as explained above, $Z$ is an $A$-invariant curve in $X$. We claim that $Z\not\subset E$.
Indeed, suppose that $Z\subset E$. Then $f(Z)=C$, since $A$ does not fix points in $C$.
Let $u$ be a non-negative real number. Then  $-K_X-uE$ is pseudoeffective if and only if $u\leqslant 1$.
Moreover, the~divisor $-K_X-uE$ is nef for $u\in[0,1]$.
Therefore, integrating $(-K_X-uE)^3=2u^3-6u^2-18u+22$, we obtain $S_X(E)=\frac{23}{44}$.
Then, it follows from Theorem~\ref{theorem:Hamid-Ziquan-Kento-1} that
\begin{equation}
\label{equation:2-16-E}
1\geqslant\frac{A_X(\mathbf{F})}{S_X(\mathbf{F})}\geqslant \min\left\{\frac{44}{23},\frac{1}{S(W_{\bullet,\bullet}^E;Z)}\right\}
\end{equation}
where $S\big(W_{\bullet,\bullet}^E; Z\big)=\frac{3}{22}\int\limits_{0}^{1}\int\limits_0^\infty \mathrm{vol}\big((-K_X-uE)\big\vert_{E}-vZ\big)dvdu$.
On the~other hand, we have
$$
(-K_X-uE)\big\vert_{E}\sim_{\mathbb{R}}\left\{\aligned
&(1+u)\mathbf{s}+(3-u)\mathbf{l}\  \text{if $E\cong\mathbb{P}^1\times\mathbb{P}^1$}, \\
&(1+u)\mathbf{s}+4\mathbf{l}\  \text{if $E\cong\mathbb{F}_2$},
\endaligned
\right.
$$
Hence, if  $E\cong\mathbb{P}^1\times\mathbb{P}^1$, then
\begin{multline*}
S\big(W^E_{\bullet,\bullet};Z\big)\leqslant\frac{3}{22}\int\limits_0^1\int\limits_0^\infty \mathrm{vol}\big((-K_X-uE)\big\vert_{E}-v\mathbf{s}\big)dvdu=\\
=\frac{3}{22}\int\limits_0^1\int\limits_0^{1+u}\big((1+u-v)\mathbf{s}+(3-u)\mathbf{l}\big)^2dvdu=\frac{3}{22}\int\limits_0^1\int\limits_0^{1+u}2(3-u)(1+u-v)dvdu=\frac{67}{88},
\end{multline*}
because $|Z-\mathbf{s}|\ne\varnothing$, since $f(Z)=C$. Similarly, if $E\cong\mathbb{F}_2$, we estimate $S(W^E_{\bullet,\bullet};Z)\leqslant\frac{41}{44}$.
In both cases, we have $S(W^E_{\bullet,\bullet};Z)<1$, which contradicts \eqref{equation:2-16-E}.
This shows that $Z\not\subset E$ as claimed.

Since $A$ does not fix points in $\Delta_4$, we see that $\pi(Z)$ is not a point in the~curve $\Delta_4$.
Then
\begin{enumerate}
\item either $\pi(Z)$ is a point in $\mathbb{P}^2$ and $Z$ is a smooth fiber of the~conic bundle $\pi$,
\item or $\pi(Z)$ is a curve in $\mathbb{P}^2$, which can be an irreducible component of the~curve $\Delta_4$.
\end{enumerate}
In both cases, let $p$ be a sufficiently general point in $Z$, let $\mathscr{C}$ be the~fiber of the~conic bundle $\pi$ that contains $p$,
and let $S$ be a general surface in $|H-E|$ that contains $\mathscr{F}$.
Then $\mathscr{C}$ is reduced, which implies that the~surface $S$ is smooth.
Note that $p\not\in E$, since $Z\not\subset E$.
Recall that $\delta_p(X)\leqslant 1$ by assumption.

We claim that $S$ is a smooth del Pezzo surface of degree $4$.
Indeed, since$-K_S\sim H\vert_{S}$ and $(-K_S)^2=4$, the~divisor $-K_S$ is nef and big,
and the~only curves that have trivial intersection with it are the~fibers of the~natural projection $E\to C$ contained in $S$.
On the~other hand, the~intersection $E\cap S$ is an irreducible curve that is not a fiber of the~projection $E\to C$ unless $E\simeq\mathbb{F}_2$ and $\mathscr{C}=\mathbf{s}$.
The~latter possibility cannot hold in our case, since $p\in\mathscr{C}$ and $p\not\in E$.
Thus, we see that $S$ is a smooth del Pezzo surface of degree $4$.

Let us apply Theorem~\ref{theorem:Hamid-Ziquan-Kento} to $S$ to estimate $\delta_p(X)$ from below.
Since $-K_X-uS\sim_{\mathbb{R}}(2-u)H-(1-u)E$,
the~divisor $-K_X-uS$ is pseudoeffective for $u\in[0,2]$, and it is not pseudoeffective for $u>2$.
For $u\in[0,2]$, let $P(u)$ be the~positive part of the~Zariski decomposition of the~divisor $-K_X-uS$
and let $N(u)$ be its negative part. Then
$$
P(u)=
\left\{\aligned
&(2-u)H-(1-u)E \ \text{if $0\leqslant u\leqslant 1$}, \\
&(2-u)H\ \text{if $1\leqslant u\leqslant 2$},
\endaligned
\right.
$$
and
$$
N(u)=\left\{\aligned
&0 \ \text{if $0\leqslant u\leqslant 1$}, \\
&(u-1)E\ \text{if $1\leqslant u\leqslant 2$}.
\endaligned
\right.
$$
Thus, integrating, we get $S_X(S)=\frac{1}{22}\int\limits_{0}^{2}\big(P(u)\big)^3du=\frac{1}{22}\int\limits_{0}^{1}6u^2-24u+22du+\frac{1}{22}\int\limits_{1}^{2}4(2-u)^3du=\frac{13}{22}$,
so it follows from Theorem~\ref{theorem:Hamid-Ziquan-Kento} that
\begin{equation}
\label{equation:flag-3-Abban-Zhuang}
\delta_p(X)\geqslant\min\Bigg\{\frac{22}{13},\inf_{\substack{F/S\\p\in C_S(F)}}\frac{A_S(F)}{S\big(W^S_{\bullet,\bullet};F\big)}\Bigg\},
\end{equation}
where $S\big(W^S_{\bullet,\bullet};F\big)=\frac{3}{22}\int\limits_0^2\int\limits_0^\infty \mathrm{vol}\big(P(u)\big\vert_{S}-vF\big)dvdu$, since $p\not\in E$,
and the~infimum in \eqref{equation:flag-3-Abban-Zhuang} is taken by all prime divisors over $S$ whose center on $S$ contains $p$.
To estimate this infimum, we note that
$$
P(u)\big\vert_{S}\sim_{\mathbb{R}}
\left\{\aligned
&-K_S+(1-u)\mathscr{C} \ \text{if $0\leqslant u\leqslant 1$}, \\
&(2-u)(-K_S)\ \text{if $1\leqslant u\leqslant 2$},
\endaligned
\right.
$$
we set $D_u=-K_S+(1-u)\mathscr{F}$, and we let
$$
\delta_p(S,D_u)=\inf_{\substack{F/S\\ p\in C_S(F)}}\frac{A_{S}(F)}{S_{D_u}(F)},
$$
where the~infimum is taken over all prime divisors $F$ over the~surface $S$ whose center on $S$ contains $p$,
and
$$
S_{D_u}(F)=\frac{1}{D_u^2}\int\limits_0^\infty \mathrm{vol}\big(D_u-vF\big)du=\frac{1}{8-4u}\int\limits_0^\infty \mathrm{vol}\big(D_u-vF\big)du.
$$
Then if follows from \cite[Lemma 23]{CheltsovFujitaKishimotoOkada} that
\begin{equation}
\label{equation:dP4-C-delta}
\delta_p(S,D_u)\geqslant
\left\{\aligned
&\frac{24}{u^2-10u+28}\ \text{if $\mathscr{C}$ is irreducible}, \\
&\frac{48-24u}{12u^2-54u+61}\ \text{if $\mathscr{C}$ is reducible},
\endaligned
\right.
\end{equation}
and it follows from \cite[Lemma 2.12]{Book} that $\delta_p(S,D_1)=\delta_p(S)\geqslant\delta(S)\geqslant\frac{4}{3}$.
Thus, if $\mathscr{C}$ is irreducible, then
\begin{multline*}
S\big(W^S_{\bullet,\bullet};F\big)=\frac{3}{22}\int\limits_0^1\int\limits_0^\infty \mathrm{vol}\big(-K_S+(1-u)C-vF\big)dvdu+\frac{3}{22}\int\limits_1^2\int\limits_0^\infty \mathrm{vol}\big((2-u)(-K_S)-vF\big)dvdu\leqslant\\
\leqslant\frac{3}{22}A_S(F)\int\limits_0^1\frac{(8-4u)(u^2-10u+28)}{24}du+\frac{3}{22}\int\limits_1^2\int\limits_0^\infty \mathrm{vol}\big((2-u)(-K_S)-vF\big)dvdu=\\
=\frac{13}{16}A_S(F)+\frac{3}{22}\int\limits_1^2\int\limits_0^\infty (2-u)^2\mathrm{vol}\big(-K_S-\frac{v}{2-u}F\big)dvdu=\\
=\frac{13}{16}A_S(F)+\frac{3}{22}\int\limits_1^2(2-u)^3\int\limits_0^\infty\mathrm{vol}\big(-K_S-vF\big)dvdu\leqslant\frac{13}{16}A_S(F)+\frac{3}{22}\int\limits_1^2(2-u)^3\frac{4}{\delta_p(S)}A_S(F)du\leqslant \\
\leqslant\frac{13}{16}A_S(F)+\frac{3}{22}\int\limits_1^23(2-u)^3A_S(F)du=\frac{13}{16}A_S(F)+\frac{9}{88}A_S(F)=\frac{161}{176}A_S(F)
\end{multline*}
for every prime divisor $F$ over $S$ such that $p\in C_S(F)$, so it follows from \eqref{equation:flag-3-Abban-Zhuang} that $\delta_p(X)\geqslant\frac{176}{161}$ in this case.
Similarly, if $\mathscr{C}$ is reducible, get $\delta_p(X)\geqslant\frac{88}{85}$.
So, $\delta_p(X)>1$ in both cases, which is a contradiction.
\end{proof}

\begin{lemma}
\label{lemma:pointless-2-21}
Let $X$ be a smooth Fano 3-fold in Family \textnumero 2.21, and let $A$ be a finite abelian subgroup in $\mathrm{Aut}(X)$ that does not fix points on $X$.
Then $X$ is K-polystable.
\end{lemma}

\begin{proof}
The required assertion follows from \cite{Malbon,Malbon2024}.
Indeed, it follows from  \cite{Malbon} that there exists an involution $\tau\in\mathrm{Aut}(X)$ such that such that $\mathrm{Pic}(X)^{\langle\tau\rangle}=\mathbb{Z}[-K_X]$, and we have the following commutative diagram:
$$
\xymatrix@R=1em{
X\ar@{->}[d]_{\pi}\ar@{->}[rr]^{\tau}&&Q\ar@{->}[d]^{\pi}\\%
Q\ar@{-->}[rr]_{\iota}&&Q&}
$$
where $Q$ is a smooth quadric 3-fold in $\mathbb{P}^4$,
$\pi$ is the blowup of a smooth rational twisted quartic curve $C$,
and $\iota$ is a birational involution given by the linear subsystem in $|\mathcal{O}_Q(2)|$ consisting of surfaces passing through $C$. 
Moreover, since $\pi$ is $\mathrm{Aut}(Q;C)$-equivariant, we can identify $\mathrm{Aut}(Q;C)$ with a subgroup of the group $\mathrm{Aut}(X)$.
Then we can choose $\tau\in\mathrm{Aut}(X)$  such that $\tau$ commutes with every element in $\mathrm{Aut}(Q;C)$, 
so, in particular, we have
$$
\mathrm{Aut}(X)\simeq \mathrm{Aut}(Q;C)\times\mathbb{Z}/2\mathbb{Z}.
$$
The paper \cite{Malbon} also lists all possibilities for the group $\mathrm{Aut}(Q;C)$.
Furthermore, it follows from \cite{Malbon2024} that $X$ is K-pollystable if and only if $\mathrm{Aut}(Q;C)$ contains a subgroup isomorphic to $(\mathbb{Z}/2\mathbb{Z})^2$. Now, it is very easy to deduce from  \cite{Malbon}  that $X$ is K-polystable.

Alternatively, we can prove the required assertion using only technical results obtained in \cite{Malbon2024}.
Namely, suppose that $X$ is not is K-polystable. Let $\mathbf{F}$ be an $A$-invariant divisor provided by Lemma~\ref{lemma:F},
let~$Z$ be its center on $X$, and let $p$ be a general point in $Z$. Then, as explained above, $Z$ is an irreducible $A$-invariant curve,
and $\delta_p(X)\leqslant 1$.
Then it follows from \cite[Technical~Theorem 1]{Malbon2024} that $p$ is contained in the union of the~exceptional surfaces of the~blowups $\pi$ and $\pi^\prime$.
Without loss of generality, we may assume that $Z$ is contained in the~$\pi$-exceptional surface.
Then it follows from \cite[Technical~Theorem~2]{Malbon2024} that $\pi(Z)$ is a point,
which implies that $\pi^\prime(Z)$ is a line.
Hence, the~group $A$ preserves both extremal rays of the~Mori cone $\overline{\mathrm{NE}}(X)$,
and the~diagram above must be $A$-equivariant.
Then $A$ fixes the~point $\pi(Z)$, which is impossible by Theorem~\ref{theorem:Kollar-Szabo}.
\end{proof}

\begin{lemma}
\label{lemma:pointless-2-24}
Let $X$ be a smooth Fano 3-fold in Family \textnumero 2.24, and let $A$ be a finite abelian subgroup in $\mathrm{Aut}(X)$ that does not fix points on $X$.
Then $X$ is K-polystable.
\end{lemma}

\begin{proof}
Suppose that $X$ is not K-polystable.
Then it follows from \cite[\S~4.7]{Book} that $X$ is a smooth divisor of degree $(1,2)$ in $\mathbb{P}^2_{x,y,z}\times\mathbb{P}^2_{u,v,w}$
such that either
\begin{equation}
\label{equation:2-24-unstable-finite-Aut}
X=\big\{\big(vw+u^2\big)x+\big(uw+v^2\big)y+w^2z=0\big\},
\end{equation}
or
\begin{equation}
\label{equation:2-24-unstable-infinite-Aut}
X=\big\{\big(vw+u^2\big)x+v^2y+w^2z=0\big\}.
\end{equation}
Let $\mathrm{pr}_1\colon X\to \mathbb{P}^2$ and $\mathrm{pr}_2\colon X\to \mathbb{P}^2$ be the~projections to the~first and the~second factors, respectively.
Then $\mathrm{pr}_1$ is an $A$~equivariant conic bundle, and $\mathrm{pr}_2$ is an~$A$-equivariant $\mathbb{P}^1$-bundle.
Let $\Delta_3$ be the~discriminant curve of the~conic bundle $\mathrm{pr}_1$. Then  $\Delta_3$ is a~cubic curve with at worst nodal singularities.

We claim that $A$ does not fix points in $\Delta_3$. Indeed, suppose $A$ fixes a point $p\in\Delta_3$.
Let $F$ be the~fiber of the~conic bundle $\mathrm{pr}_1$ over $p$ taken with reduced structure.
Then $F$ is $A$-invariant. Thus, if $p$ is a smooth point of $\Delta_3$, then $F$ has one singular point, and this point must be fixed by $A$.
Similarly, if $p\in\mathrm{Sing}(\Delta_3)$, then $\mathrm{pr}_2(F)$ is a $A$-invariant line in $\mathbb{P}^2_{u,v,w}$,
so Lemma~\ref{lemma:lift} implies that $A$ fixes a point in $\mathrm{pr}_2(F)$,
so that $A$ fixes a point in $F$, since $\mathrm{pr}_2$ induces an $A$-equivariant isomorphism $F\to\mathrm{pr}_2(F)$.
Thus, in both cases the group $A$ fixes a point in $F$ which is impossible by assumption.
Hence, we conclude that $A$ does not fix points in $\Delta_3$. If $X$ is given by \eqref{equation:2-24-unstable-finite-Aut}, then $\Delta_3$ is an irreducible cubic curve with one singular point,
so $A$ fixes the~singular point of $\Delta_3$, which is impossible.
Similarly, if $X$ is given by \eqref{equation:2-24-unstable-infinite-Aut}, then $\Delta_3=L+C$, where $L$ is a line, and $C$ is a smooth conic.
In this case, applying Lemma~\ref{lemma:lift} to $\mathbb{P}^2_{x,y,z}$, we see that $A$ fixes a point in $p\in L$,
which we proved to be impossible.
\end{proof}

\begin{lemma}
\label{lemma:pointless-3-2}
Let $X$ be a smooth Fano 3-fold in Family \textnumero 3.2, and let $A$ be a finite abelian subgroup in $\mathrm{Aut}(X)$ that does not fix points on $X$.
Then $X$ is K-polystable.
\end{lemma}

\begin{proof}
It is well-known that $-K_X$ is very ample and gives $A$-equivariant embedding $X\hookrightarrow\mathbb{P}^9$,
and quadrics in $\mathbb{P}^9$ containing the~image cuts out a 4-fold $V\subset\mathbb{P}^9$ such that
$V=\mathbb{P}\big(\mathcal{O}_{\mathbb{P}^1}(2)\oplus\mathcal{O}_{\mathbb{P}^1}(2)\oplus\mathcal{O}_{\mathbb{P}^1}(1)\oplus\mathcal{O}_{\mathbb{P}^1}(1)\big)$,
and $X$ is contained in the~linear system~$|3M-4F|$,
where $M$ is the~tautological line bundle on $V$,
and $F$ is a fiber of the~natural projection~$V\to\mathbb{P}^1$.
In the~notation of \cite[\S2]{Reid}, we have $V=\mathbb{F}(2,2,1,1)$, and $X$ is given by the~following equation:
\begin{multline*}
\quad\quad\quad\alpha^1_2(t_1,t_2)x_1^3+\alpha^2_2(t_1,t_2)x_1^2x_2+\alpha^1_1(t_1,t_2)x_1^2x_3+\alpha^2_1(t_1,t_2)x_1^2x_4+\\
+\alpha^3_2(t_1,t_2)x_1x_2^2+\alpha^3_1(t_1,t_2)x_1x_2x_3+\alpha^4_1(t_1,t_2)x_1x_2x_4+\alpha^1_0(t_1,t_2)x_1x_3^2+\\
+\alpha^2_0(t_1,t_2)x_1x_3x_4+\alpha^3_0(t_1,t_2)x_1x_4^2+\alpha^4_2(t_1,t_2)x_2^3+\alpha^5_1(t_1,t_2)x_2^2x_3+\\
+\alpha^6_1(t_1,t_2)x_2^2x_4+\alpha^4_0(t_1,t_2)x_2x_3^2+\alpha^5_0(t_1,t_2)x_2x_3x_4+\alpha^6_0(t_1,t_2)x_2x_4^2=0,\quad\quad\quad
\end{multline*}
where each $\alpha^i_d(t_1,t_2)$ is a polynomial of degree $d$.

Let $\pi\colon X\to \mathbb{P}^1_{t_1,t_2}$ be the~morphism induced by the~natural projection~$V\to\mathbb{P}^1$,
and let $S$ be the~surface in $X$ that is given by $x_1=x_2=0$.
Then $\pi$ is $A$-equivariant, fibers of $\pi$ are cubic surfaces in $\mathbb{P}^3$,
the surface $S$ is $A$-invariant, $S\cong\mathbb{P}^1\times\mathbb{P}^1$,
the normal bundle of $S$ in~$X$ is $\mathcal{O}_{\mathbb{P}^1\times\mathbb{P}^1}(-1,-1)$,
and $S$ cuts a line on every fiber of $\pi$.
Moreover, by Lemma~\ref{lemma:lift} and Corollary~\ref{corollary:odd-degree-curves}, $A$ acts without fixed points on $\mathbb{P}^1_{t_1,t_2}$.

Suppose that $X$ is not K-polystable and let $\mathbf{F}$ be an $A$-invariant divisor whose existence is asserted by Lemma~\ref{lemma:F},
and let $Z$ be its center on $X$. Then, as explained above, $Z$ is an irreducible $A$-invariant curve,
and $\pi(Z)=\mathbb{P}^1_{t_1,t_2}$, since $A$ does not fix points in $\mathbb{P}^1_{t_1,t_2}$.
Applying \cite[Lemma~6.1]{BelousovLoginov2024} and \cite[Lemma~6.2]{BelousovLoginov2024},
we see that $\delta_p(X)>1$ for a general point $p\in Z$, which contradicts the existence of $\mathbf{F}$.
\end{proof}

\begin{lemma}
\label{lemma:pointless-3-5}
Let $X$ be a smooth Fano 3-fold in Family \textnumero 3.5, and let $A$ be a finite abelian subgroup in $\mathrm{Aut}(X)$ that does not fix points on $X$.
Then $X$ is K-polystable.
\end{lemma}

\begin{proof}
Set $S=\mathbb{P}^1\times\mathbb{P}^1$ and let $C\subset S$ be a~smooth curve of degree $(5,1)$ and construct $X$ as follows.
Consider the~embedding $S\hookrightarrow\mathbb{P}^1\times\mathbb{P}^2$ given by $([u:v],[x:y])\mapsto([u:v],[x^2:xy:y^2])$,
and identify $S$ and $C$ with their images in $\mathbb{P}^1\times\mathbb{P}^2$ using this embedding.
Then there exists a birational morphism $\pi\colon X\to\mathbb{P}^1\times\mathbb{P}^2$ that blows up $C$.
Moreover, it follows from Lemma~\ref{lemma:Prokhorov} that $\pi$ is $A$-equivariant, both $C$ and $S$ are $A$-invariant,
projection to the~first factor $\mathrm{pr}_1\colon\mathbb{P}^1\times\mathbb{P}^2 \to \mathbb{P}^1$ and projection to the~second factor $\mathrm{pr}_2\colon\mathbb{P}^1\times\mathbb{P}^2 \to \mathbb{P}^2$ are both $A$-equivariant.
Note that $A$ does not fix points in $\mathbb{P}^1\times\mathbb{P}^2$ by Theorem~\ref{theorem:Kollar-Szabo}.
On the~other hand, applying Lemma~\ref{lemma:lift} to $\mathbb{P}^2$ and the~conic $\mathrm{pr}_2(S)$, we see that $A$ fixes a point in $\mathbb{P}^2$.
Then $A$ does not fix points in $\mathbb{P}^1$, so $X$ is K-polystable by \cite[Corollary 5.70]{Book}.
\end{proof}

\begin{lemma}
\label{lemma:pointless-3-6}
Let $X$ be a smooth Fano 3-fold in Family \textnumero 3.6, and let $A$ be a finite abelian subgroup in $\mathrm{Aut}(X)$ that does not fix points on $X$.
Then $X$ is K-polystable.
\end{lemma}

\begin{proof}
Let us use the commutative diagram \eqref{equation:3-6} together with all notations introduced in that diagram.
Note~that \eqref{equation:3-6} is $A$-equivariant by Lemma~\ref{lemma:Prokhorov}.
Moreover, $A$ does not fix points in the~left $\mathbb{P}^1$ in \eqref{equation:3-6}.
Indeed, if $A$ were to fix a point in the~left $\mathbb{P}^1$ in \eqref{equation:3-6},
the~fiber of $\phi_L$ over this point would be $A$-invariant, so its image in $\mathbb{P}^3$ would be an $A$-invariant plane,
which would imply that $A$ also fixes a point in $\mathbb{P}^3$ by Lemma \ref{lemma:lift}, so $A$ would fix a point in $X$ by Theorem~\ref{theorem:Kollar-Szabo}, contradicting our assumption.

Now, the~K-polystability of $X$ follows from \cite{Cheltsov2024}.
Indeed, suppose that $X$ is not K-polystable.
Let $\mathbf{F}$ be an $A$-invariant divisor whose existence is asserted by Lemma~\ref{lemma:F},
and let $Z$ be its center on $X$. Then $Z$ is an irreducible $A$-invariant curve with $\varphi_L(Z)=\mathbb{P}^1$, since $A$ does not fix points in the~left $\mathbb{P}^1$ in \eqref{equation:3-6}.
Now, applying \cite[$(\bigstar)$]{Cheltsov2024} and \cite[Corollary~2]{Cheltsov2024} to a general point $p\in Z$, we obtain
$$
1\geqslant\delta(X)=\frac{A_X(\mathbf{F})}{S_X(\mathbf{F})}>1,
$$
which is absurd. This completes the~proof of the~lemma.
\end{proof}

\begin{lemma}
\label{lemma:pointless-3-10}
Let $X$ be a smooth Fano 3-fold in Family \textnumero 3.10, and let $A$ be a finite abelian subgroup in $\mathrm{Aut}(X)$ that does not fix points on $X$.
Then $X$ is K-polystable.
\end{lemma}

\begin{proof}
Suppose that the~3-fold $X$ is not K-polystable. It follows from \cite[\S 5.17]{Book} that there is a birational morphism $\pi\colon X\to Q$ such that $Q$ is a smooth quadric 3-fold in $\mathbb{P}^4$,
the morphism $\pi$ is a blowup of two disjoint smooth conics $C_1=\{x=0,y=0,w^2+zt=0\}$ and $C_2=\{z=0,t=0,w^2+xy=0\}$,
and either
\begin{equation}
\label{equation:C}
Q=\big\{w^2+xy+zt+xt+xz=0\big\},
\end{equation}
or
\begin{equation}
\label{equation:B}
Q=\big\{w^2+xy+zt+a(xt+yz)+xz=0\big\}
\end{equation}
for some $a\in\mathbb{C}$ with $a\ne\pm 1$.

The description of the~Mori cone $\overline{\mathrm{NE}}(X)$ is given in \cite[\S~III.3, p.80-p.81]{Matsuki}.
Using this description, we see that the~morphism $\pi$ is $A$-equivariant, and we have the~following $A$-equivariant commutative diagram:
$$
\xymatrix@R=1em{
&X\ar@{->}[ld]_{\pi}\ar@{->}[rd]^{\eta}&&\\%
Q\ar@{-->}[rr]_{\chi}&&\mathbb{P}^1_{x,y}\times\mathbb{P}^1_{z,t}}
$$
where $\chi$ is a rational map given by $[x:y:z:t:w]\mapsto([x:y],[z:t])$,
and $\eta$ is a (standard) conic bundle. Here, and below, we use $([x:y],[z:t])$ as coordinates on $\mathbb{P}^1_{x,y}\times\mathbb{P}^1_{z,t}$,
which may create a minor ambiguity.

Let $\Delta$ be the~discriminant curve of the~conic bundle $\eta$. Then $\Delta$ is an $A$-invariant curve of degree $(2,2)$,
which has at worst nodal singularities.
Moreover, if we are in case \eqref{equation:C}, then
$$
\Delta=\big\{x(t^2x+2txz-4tyz+xz^2)=0\big\}\subset \mathbb{P}^1_{x,y}\times\mathbb{P}^1_{z,t}.
$$
Similarly, if we are in case \eqref{equation:B}, then
$$
\Delta=\big\{a^2t^2x^2+(2a^2-4)xyzt+2atzx^2+a^2y^2z^2+2ayz^2x+z^2x^2=0\big\}\subset \mathbb{P}^1_{x,y}\times\mathbb{P}^1_{z,t},
$$
If $a\ne 0$, this curve is irreducible with a node at $([0:1],[0:1])$.
If $a=0$, then $\Delta=\{zx(zx-4yt)=0\}$, and the~curve $\Delta$ splits as a~union of curves of degrees $(0,1)$, $(1,0)$, $(1,1)$. We claim that $A$ does not fix points in $\Delta$.
Indeed, suppose that it does and let $F$ be the~fiber of the~conic bundle $\eta$ over an $A$-fixed point in $\Delta$. Then $F$ is singular.
If $F$ is reduced, then it has one singular point, and $A$ must fix this point, which contradicts our assumption.
If $F$ is not reduced, then $F=2\ell$ for a smooth $A$-invariant curve $\ell$ with $\pi(\ell)$ an $A$-invariant line in $Q$.
On the~other hand, it follows from Lemma~\ref{lemma:lift} that the~action of the~group $A$ on the~quadric $Q$ is induced by a five-dimensional faithful representation of the~group $A$,
so the~line $\pi(\ell)$ corresponds to a two-dimensional subrepresentation, which must split as a sum of two one-dimensional representations, since $A$ is abelian.
This shows that $A$ fixes a point in $\pi(\ell)$, so $A$ fixes a point in $X$ by Theorem~\ref{theorem:Kollar-Szabo},
which contradicts our assumption.

Hence, we see that $A$ does not fix point in $\Delta$. This implies that we are not in case \eqref{equation:B}.
Thus, we are in case \eqref{equation:C}. In this case, we have $\Delta=L+Z$, where $L$ is a curve of degree $(0,1)$, and $Z$ is an irreducible curve of degree $(2,1)$.
Set $S=\pi_*(\eta^*(L))$. Then $S$ is a hyperplane section of the~quadric $Q$, which is cut out on $Q$ by the~hyperplane $\{x=0\}$.
One can check that $S$ is a quadric cone with one singular point, so the~vertex of $S$ is fixed by $A$, which contradicts our assumption.
\end{proof}

\begin{lemma}
\label{lemma:pointless-3-12}
Let $X$ be a smooth Fano 3-fold in Family \textnumero 3.12, and let $A$ be a finite abelian subgroup in $\mathrm{Aut}(X)$ that does not fix points on $X$.
Then $X$ is K-polystable.
\end{lemma}

\begin{proof}
There is a birational morphism $\pi\colon X\to \mathbb{P}^3$ which is the~blowup of a line $L$ and a twisted cubic~$C$, which are disjoint. 
Let $f\colon V\to \mathbb{P}^3$ be the~blowup of the~curve $L$, and let $\widetilde{C}$ be the~strict transform of the~curve $C$ on  $V$.
Then there exists a Sarkisov link:
$$
\xymatrix@R=1em{
&V\ar@{->}[ld]_{f}\ar@{->}[rd]^{g}&&\\%
\mathbb{P}^3&&\mathbb{P}^1}
$$
where $g$ is a $\mathbb{P}^2$-bundle over $\mathbb{P}^1$.
Moreover, the~map $g$ induces a finite morphism $\omega\colon \widetilde{C}\to \mathbb{P}^1$ of degree $3$. Note that the description in \cite[\S~III.3, p.83-p.84]{Matsuki} implies that the blowup $X ={\rm Bl}_{\widetilde{C}}(V) \to V$ is $A$-equivariant, and both $f$ and $g$ are $A$-equivariant.
Furthermore, it follows from \cite{Denisova} that $X$ is not K-polystable if and only if the~triple cover $\omega\colon \widetilde{C}\to \mathbb{P}^1$  has a unique ramification point of ramification index $3$.
Thus, if $X$ is not K-polystable, then the~ramification point of index $3$ of the~finite morphism $\omega$ is fixed by $A$.
Hence, if $X$ is not K-polystable, then $A$ fixes a point in $X$ by Theorem~\ref{theorem:Kollar-Szabo}, which contradicts our assumption.
\end{proof}

\begin{lemma}
\label{lemma:pointless-3-13}
Let $X$ be a smooth Fano 3-fold in Family \textnumero 3.13, and let $A$ be a finite abelian subgroup in $\mathrm{Aut}(X)$ that does not fix points on $X$.
Then $X$ is K-polystable.
\end{lemma}

\begin{proof}
Suppose that $X$ is not K-polystable. Then it follows from \cite[Lemma~5.97]{Book}, \cite[Lemma~5.98]{Book} and
\cite[Proposition~5.99]{Book} that $\mathrm{Aut}(X)\simeq\mathbb{G}_a\rtimes\mathfrak{S}_3$,
which implies that either $A$ is trivial,  $A\simeq\mathbb{Z}/2\mathbb{Z}$, or $A\simeq\mathbb{Z}/3\mathbb{Z}$, which is impossible by Remark~\ref{remark:cyclic-Lefschetz}.
\end{proof}

\begin{lemma}
\label{lemma:pointless-4-13}
Let $X$ be a smooth Fano 3-fold in Family \textnumero 4.13, and let $A$ be a finite abelian subgroup in $\mathrm{Aut}(X)$ that does not fix points on $X$.
Then $X$ is K-polystable.
\end{lemma}

\begin{proof}
In this case, there exists is a birational morphism $\pi\colon X\to\mathbb{P}^1_{x_0,x_1}\times\mathbb{P}^1_{y_0,y_1}\times\mathbb{P}^1_{z_0,z_1}$ 
which is the~blowup of a smooth rational curve $C$ such that $C$ is a curve of degree $(1,1,3)$. 
Suppose that $X$ is not K-polystable. Then it follows from \cite[\S~5.22]{Book} that, up to a change of coordinates, we have 
$$
C=\big\{x_0y_1-x_1y_0=x_0^3z_0+x_1^3z_1+x_0x_1^2z_0=0\big\}.
$$
Moreover, it follows from the~description of the~extremal rays of the~Mori cone $\overline{\mathrm{NE}}(X)$ given in \cite[\S~III.3, p.84-p.85]{Matsuki}
that $\pi$ and the~projection $X\to \mathbb{P}^1_{z_0,z_1}$ given by $([x_0,x_1],[y_0,y_1],[z_0,z_1])\mapsto[z_0,z_1]$ are both $A$-equivariant, so $C$ is $A$-invariant.
On the~other hand, the~projection $X\to \mathbb{P}^1_{z_0,z_1}$ induces a triple cover $C\to\mathbb{P}_{z_0,z_1}$,
which has a unique ramification point with ramification index $3$,
so this point must be fixed by $A$, which contradicts out assumption.
\end{proof}

\section{Examples of smooth Fano 3-folds not satisfying Condition~$\mathbf{(A)}$}
\label{section:pointless}

\setlength{\epigraphwidth}{0.8\textwidth}
\epigraph{What is all this games stuff, anyway? In a way, it's easy enough to see what happens to winners who stay on Earth. They get all the~juicy jobs; they become governors and such.}{\emph{The World of Null-A}, A.\ E.\ van Vogt, 1948}

In this section, we provide examples of smooth Fano 3-folds that do not satisfy Condition~$\mathbf{(A)}$.

\subsection{K-unstable smooth Fano 3-folds not satisfying Condition~$\mathbf{(A)}$.}

\begin{example}
\label{example:0}
Let $X=\mathbb{P}^1\times S$, where $S$ is a blow up of $\mathbb{P}^2$ in one or two points.
Then $\mathrm{Aut}(X)$ contains a subgroup isomorphic to $(\mathbb{Z}/2\mathbb{Z})^2$ that acts trivially on the~second factor, and does not fix points in $X$.
\end{example}

\begin{example}
\label{example:1}
Let $X$ be the~blowup of a smooth quadric 3-fold in $Q\subset\mathbb{P}^4$ along a smooth quartic elliptic curve $C_4$.
Choosing appropriate coordinate on $\mathbb{P}^4$, we may assume that $Q=\{x_1^2+x_2^2+x_3^2+x_4^2+x_5^2=0\}$,
and the~curve $C_4$ is given by the~following equations:
$$
x_5=\sum_{i=1}^{5}x_i^2=\sum_{i=1}^{5}\lambda_ix_i^2=0
$$
for some $\lambda_1,\ldots,\lambda_5\in\mathbb{C}$. Let $A$ be the~abelian subgroup in $\mathrm{Aut}(Q)$ consisting of projective transformations that change signs of coordinates $x_1,\ldots,x_5$. Then $A\simeq(\mathbb{Z}/2\mathbb{Z})^4$, the~group $A$ does not fix points in $Q$, and the~curve $C_4$ is $A$-invariant, so the~action of the~group $A$ lifts to $X$, and $A$ does not fix points in $X$.
\end{example}

\begin{example}
\label{example:2}
Let $Q$ be the~quadric cone in $\mathbb{P}^4$ with one singular point, and let $f\colon X\to Q$ be the~blowup of its vertex.
Then choosing appropriate coordinate on $\mathbb{P}^4$, we may assume that $Q=\{x_1^2+x_2^2+x_3^2+x_4^2=0\}$.
Let $A$ be the~abelian subgroup in $\mathrm{Aut}(Q)$ consisting of projective transformations that change signs of coordinates $x_1,\ldots,x_5$.
Then $A\simeq(\mathbb{Z}/2\mathbb{Z})^4$, the~only $A$-fixed point in $Q$ is its vertex $[0:0:0:0:1]$,
and the~action of the~group $A$ lifts to $X$.
Note that there exists $A$-equivariant commutative diagram
$$
\xymatrix@R=1em{
&X\ar@{->}[ld]_{f}\ar@{->}[rd]^{\pi}&&\\%
Q\ar@{-->}[rr]_{\chi}&&S}
$$
where $S$ is the~quadric surface in $\mathbb{P}^3$ given by $x_1^2+x_2^2+x_3^2+x_4^2=0$ equipped with the~same action of the group $A$,
$\pi$ is a $\mathbb{P}^1$-bundle, and $\chi$ is given by $[x_1:x_2:x_3:x_4:x_5]\mapsto [x_1:x_2:x_3:x_4]$.
Since $A$ does not fix points in $S$, we see that $A$ does not fix points in $X$ either.
\end{example}

\begin{example}
\label{example:3}
Let $X$ be the~blowup of $\mathbb{P}^1_{x_1,y_1}\times \mathbb{P}^1_{x_2,y_2}\times\mathbb{P}^1_{x_3,y_3}$ along a smooth curve $C$ of degree $(0,1,1)$.
Then, changing coordinates if necessarily, we may assume that $C=\{y_1=0,x_2y_3-x_3y_2=0\}$.
Let $A$ be the~subgroup in $\mathrm{Aut}(\mathbb{P}^1_{x_1,y_1}\times \mathbb{P}^1_{x_2,y_2}\times\mathbb{P}^1_{x_3,y_3})$ generated by the~following transformations:
\begin{align*}
([x_1:y_1],[x_2:y_2],[x_3:y_3])&\mapsto([x_1:y_1],[x_2:-y_2],[x_3:-y_3]),\\
([x_1:y_1],[x_2:y_2],[x_3:y_3])&\mapsto([x_1:y_1],[y_2:x_2],[y_3:x_3]).
\end{align*}
Then $A\simeq (\mathbb{Z}/2\mathbb{Z})^2$, it does not fix points in  $\mathbb{P}^1_{x_1,y_1}\times \mathbb{P}^1_{x_2,y_2}\times\mathbb{P}^1_{x_3,y_3}$,
and it leaves the~curve $C$ invariant.
Hence, the~action of the~group $A$ lifts to $X$, and $A$ does not fix points in $X$.
\end{example}

\begin{example}
\label{example:4}
Let $A$ be the subgroup in $\mathrm{Aut}(\mathbb{P}^3)$ generated by 
\begin{align*}
[x_1:x_2:x_3:x_4]&\mapsto[x_1:-x_2:x_3:-x_4],\\
[x_1:x_2:x_3:x_4]&\mapsto[x_2:x_1:x_4:x_3].
\end{align*}
Then $A\simeq(\mathbb{Z}/2\mathbb{Z})^2$ and it  does not fix points in $\mathbb{P}^3$, and $A$ leaves invariant the~lines $L_1=\{x_1=x_2=0\}$ and
$L_2=\{x_3=x_4=0\}$. Blowing up $\mathbb{P}^3$ along $L_1$, we obtain the~exceptional case (6) in Section~\ref{subsection:Main-Theorem}.
Blowing up the~points $[0:0:1:0]$ and $[0:0:0:1]$, and then blowing up the~strict transform of the~line $L_1$, we obtain the~exceptional case (7).
Finally, blowing up $[0:0:1:0]$ and $[0:0:0:1]$, and then blowing up the~strict transforms of the~lines $L_1$ and $L_2$, we obtain the~exceptional case (8). In each case, the~action of the~group $A$ lifts to the~resulting Fano 3-fold, and $A$ does not fix points in it.
\end{example}

\subsection{K-polystable smooth Fano 3-folds not satisfying Condition~$\mathbf{(A)}$.}

\begin{example}
\label{example:1-9}
Let $X_3$ be the~cubic 3-fold $\{x_1x_2x_3+(x_1+x_2+x_3)x_4x_5+2025(x_4^3+x_5^4)=0\}\subset\mathbb{P}^4$,
and let $A$ be the~subgroup in $\mathrm{Aut}(X_3)$ generated by the following transformations:
\begin{align*}
[x_1:x_2:x_3:x_4:x_5]&\to [x_2:x_3:x_1:x_4:x_5],\\
[x_1:x_2:x_3:x_4:x_5]&\to [x_1:x_2:x_3:\zeta_3x_4:\zeta_3^2x_5],
\end{align*}
where $\zeta_3$ is a primitive cube root of unity.
Then $X_3$ has three ordinary double point singularities (nodes) at $[1:0:0:0:0]$, $[0:1:0:0:0]$, $[0:0:1:0:0]$,
$A\simeq(\mathbb{Z}/3\mathbb{Z})^2$, and the group $A$ does not fix points in $X_3$.
Moreover, it follows from \cite[Section 4]{CheltsovTschinkelZhang2024} that we have the following $A$-equivariant commutative diagram:
$$
\xymatrix@R=1em{
&\widetilde{X_3}\ar@{->}[rd]_{\alpha}\ar@{->}[ldd]_{\varphi}\ar@{-->}[rr]^{\zeta}&&\widetilde{V}\ar@{->}[rdd]^{\pi}\ar@{->}[ld]^{\beta} &\\%
&&V&&\\
X_3\ar@{-->}[rrrr]^{\chi}&&&&\mathbb{P}^1_{x_4:x_5}}
$$
where $\varphi$ is the blowup of the nodes of $X_3$,
$\alpha$ is the small birational morphism that contracts the~strict transform of three lines in $X_{3}$ that contain two nodes of $X_3$,
$\zeta$ is a pseudoautomorphism that flops curves contracted by $\alpha$,
$V$ is a Fano 3-fold with three nodes such that $\mathrm{Pic}(V)\simeq\mathbb{Z}$ and $(-K_V)^3=18$,
$\beta$ is a small birational morphism,
$\pi$ is a fibration into del Pezzo surfaces of degree $6$,
and $\chi$ is a rational map that is given by $[x_1:x_2:x_3:x_4:x_5]\mapsto[x_4:x_5]$.
Moreover, the abelian group $A$ does not fix points in~$V$.
Furthermore, by Namikawa \cite{Namikawa97}, we have $H^2(V,T_V)=0$, so $V$ is smoothable.
Infinitesimal deformations are parametrized by the~complex vector space $H^1(V,T_V)$ on which $A$ acts linearly.
As $A$ is abelian, this representation splits and we can choose an $A$-equivariant smoothing by \cite{Rim}, cf. \cite[Lemma~3.7]{FantechiPardini}. Let $X$ be a general such smoothing. Then $A$ does not fix points on $X$ as well, and $X$ is a smooth Fano 3-fold belonging to Family \textnumero 1.9 which does not satisfy Condition~$\mathbf{(A)}$.
\end{example}

\begin{example}[{\cite[Example~4.25]{Book}}]
\label{example:2-5}
Consider the~smooth cubic 3-fold $V=\{x_1^3+x_2^3+x_3^3+x_4^3+x_5^3=0\}\subset\mathbb{P}^4$,
let $C$ be the~smooth cubic curve in $V$ that is cut out by the~plane $\{x_1=x_2=0\}$,
and let $\pi\colon X\to V$ be the~blowup of the~curve $C$,
where $x_1$, $x_2$, $x_3$, $x_4$, $x_5$ are coordinates on $\mathbb{P}^4$.
Then $X$ is a~smooth Fano $3$-fold in Family \textnumero 2.5.
Let~$A$ be the~abelian subgroup in $\mathrm{Aut}(V)$ that is generated by transformations
$$
[x_1:x_2:x_3:x_4:x_5]\mapsto [\zeta_3^{a_1}x_1:\zeta_3^{a_2}x_2:\zeta_3^{a_3}x_3:\zeta_3^{a_4}x_4:x_5],
$$
where $\zeta_3$ is a primitive cube root of unity, and each $a_i\in\{0,1,2\}$.
Then $A\simeq(\mathbb{Z}/3\mathbb{Z})^4$, and $A$ does not fix points in $V$.
Moreover, since $C$ is $A$-invariant, the~action of the~group $A$ lifts to $X$, so we can identify $A$ with a subgroup in $\mathrm{Aut}(X)$.
Then $A$ does not fix points in $X$ by Theorem~\ref{theorem:Kollar-Szabo},
\end{example}

\begin{example}
\label{example:2-10}
In $\mathbb{P}^5$, we consider the following complete intersection of two quadrics:
$$
V=\Big\{\sum_{i=1}^6 x_i^2 = \sum_{i=1}^6 a_i x_i^2 =0\Big\}\subset\mathbb{P}^5,
$$
where $a_1,a_2,a_3,a_4,a_5,a_6$ are general complex numbers. Let $C=\{x_1=x_2=0\}\cap V$,
and let $A$ be the~subgroup in $V$ generated by the following involutions:
\begin{align*}
[x_1:x_2:x_3:x_4:x_5:x_6]&\mapsto [-x_1:x_2:x_3:x_4:x_5:x_6],\\
[x_1:x_2:x_3:x_4:x_5:x_6]&\mapsto [x_1:-x_2:x_3:x_4:x_5:x_6],\\
[x_1:x_2:x_3:x_4:x_5:x_6]&\mapsto [x_1:x_2:-x_3:x_4:x_5:x_6],\\
[x_1:x_2:x_3:x_4:x_5:x_6]&\mapsto [x_1:x_2:x_3:-x_4:x_5:x_6],\\
[x_1:x_2:x_3:x_4:x_5:x_6]&\mapsto [x_1:x_2:x_3:x_4:-x_5:x_6].
\end{align*}
Then $V$ is smooth, $C$ is smooth and $A$-invariant, $A\simeq (\mathbb{Z}/2\mathbb{Z})^5$, the group $A$ does not fix points in $V$.
Now, let $\pi\colon X\to V$ be the~blowup of the~curve $C$.
Then $X$ is a smooth Fano 3-fold in Family \textnumero 2.10 on which the~lifted action of $A$ does not fix points.
\end{example}

\begin{example}
\label{example:2-12}
In \cite[Section~2.1]{CheltsovLiMauPinardin}, there is an explicit example of a smooth sextic curves $C_6\subset\mathbb{P}^3$ of genus $3$ such that $\mathrm{Aut}(\mathbb{P}^3,C_6)$ contains a subgroup $G\simeq\mathfrak{S}_4$. The generators of the~subgroup $G$ are listed in the~proof of
\cite[Lemma~10]{CheltsovLiMauPinardin}. Let $A$ be the~unique subgroup in $G$ that is isomorphic to $(\mathbb{Z}/2\mathbb{Z})^2$. Then $A$ does not fix points in $\mathbb{P}^3$. By blowing up $\mathbb{P}^3$ along $C_6$ we obtain a smooth Fano 3-fold $X$ in Family \textnumero 2.12 with $\mathrm{Aut}(X)$ containing a subgroup  isomorphic to $(\mathbb{Z}/2\mathbb{Z})^2$ that does not fix points in $X$.
\end{example}

\begin{example}
\label{example:2-16}
Let $C=\{x_1^2+x_2^2+x_3^2=0,x_4=0,x_5=0,x_6=0\}\subset \mathbb{P}^5$,
let $V$ be the~complete intersection of two quadrics in $\mathbb{P}^5$ that is given by
$$
\left\{\aligned
&x_1^2+x_2^2+x_3^2+x_4^2+x_5^2+x_6^2=0, \\
&1983x_1x_4+1973x_2x_5+1967x_3x_6=0,
\endaligned
\right.
$$
and let $A$ be the~subgroup in $\mathrm{Aut}(V)$ that is generated by the~following involutions:
\begin{align*}
[x_1:x_2:x_3:x_4:x_5:x_6]&\mapsto [-x_1:x_2:x_3:-x_4:x_5:x_6],\\
[x_1:x_2:x_3:x_4:x_5:x_6]&\mapsto [x_1:-x_2:x_3:x_4:-x_5:x_6],\\
[x_1:x_2:x_3:x_4:x_5:x_6]&\mapsto [x_2:x_1:x_3:x_5:x_4:x_6].
\end{align*}
Then $C\subset V$, both $C$ and $V$ are smooth,
$A\simeq(\mathbb{Z}/2\mathbb{Z})^3$, $A$ does not fix points in $V$, and $C$ is $A$-invariant.
Let $\pi\colon X\to V$ be the~blowup of the~conic $C$.
Then $X$ is a smooth Fano 3-fold in Family \textnumero 2.16,
the~action of the~group $A$ lifts to $X$, and $A$ does not fix points in $X$.
\end{example}

\begin{example}
\label{example:2-21}
Let $Q$ be the~quadric 3-fold in $\mathbb{P}^4$ given by $4x_2x_4-x_1x_5-3x_3^2=0$,
and let $C$ be the twisted quartic curve in $\mathbb{P}^4$ that is given in the~parametric form as $[u:v]\mapsto[u^4:u^3v:u^2v^2:uv^3:v^4]$,
where $[u:v]\in\mathbb{P}^1$.
Then $Q$ is smooth and contains $C$.
Let $\pi\colon X\to Q$ be the~blowup of the~curve~$C$.
Then $X$ is a smooth Fano 3-fold in Family \textnumero 2.21,
and it follows from \cite[\S~5.9]{Book} or from \cite{Malbon} that
the~group $\mathrm{Aut}(X)$ contains an~involution $\sigma$ creating the~following commutative diagram:
$$
\xymatrix@R=1em{
X\ar@{->}[d]_{\pi}\ar@{->}[rr]^{\sigma}&&X\ar@{->}[d]^{\pi}\\%
Q\ar@{-->}[rr]^{\tau}&&Q&}
$$
where $\tau$ is a~birational involution of the~quadric 3-fold $Q$ that is given by
$$
[x_1:x_2:x_3:x_4:x_5]\mapsto\big[4(x_1x_3-x_2^2): 2(x_1x_4-x_2x_3):(x_1x_5-x_3^2):2(x_2x_5-x_3x_4):4(x_3x_5-x_4^2)\big].
$$
Now, we let $A$ be the~subgroup in $\mathrm{Aut}(X)$ generated by $\sigma$ and the~lifts of the~following two transformations:
\begin{align*}
[x_1:x_2:x_3:x_4:x_5]&\mapsto [x_5:x_4:x_3:x_2:x_1],\\
[x_1:x_2:x_3:x_4:x_5]&\mapsto [x_1:-x_2:x_3:-x_4:x_5].
\end{align*}
Then $A\simeq(\mathbb{Z}/2\mathbb{Z})^3$, and $A$ does not fix points in $X$.
\end{example}


\begin{example}
\label{example:2-24}
Let $X$ be the~divisor in $\mathbb{P}^2_{x,y,z}\times\mathbb{P}^2_{u,v,w}$ that is given by the~following equation:
$$
xu^2+yv^2+zw^2+\mu\big(xvw+yuw+zuv\big)=0,
$$
where $\mu\in \mathbb{C}$ such that $\mu^3\ne -1$. Then $X$ is a smooth Fano 3-fold in Family \textnumero 2.24,
and every K-polystable smooth member in Family \textnumero 2.24 can be obtained in this way \cite[\S~4.7]{Book}.
Let $A$ be the~subgroup in $\mathrm{Aut}(X)$ that is generated by the~following transformations:
\begin{align*}
\big([x:y:z],[u:v:w]\big)&\mapsto \big([y:z:x],[v:w:u)\big]\big),\\
\big([x:y:z],[u:v:w]\big)&\mapsto \big([\zeta_3 x:\zeta_3^2 y:z],[\zeta_3 u:\zeta_3^2 v:w]\big).
\end{align*}
Then $A\simeq(\mathbb{Z}/3\mathbb{Z})^2$ and $A$ does not fix points in $X$.
\end{example}

\begin{example}
\label{example:3-2}
Let $V=\mathbb{P}\big(\mathcal{O}_{\mathbb{P}^1}(2)\oplus\mathcal{O}_{\mathbb{P}^1}(2)\oplus\mathcal{O}_{\mathbb{P}^1}(1)\oplus\mathcal{O}_{\mathbb{P}^1}(1)\big)$,
let $M$ be the~tautological line bundle on $V$,
and let $F$ be a fiber of the~natural projection $V\to\mathbb{P}^1$.
In the~notations used in \cite[\S2]{Reid}, let $X$ be the~3-fold in the~linear system $|3M-4F|$ that is given by the~following equation:
\begin{multline*}
\quad \quad \quad \quad \quad x_1x_3x_4+3x_2(x_3^2+x_4^2)+1007t_1t_2x_1^3-301t_1t_2x_1x_2^2+\\
+(t_1^2+t_2^2)(71x_1^2x_2-54x_2^3)+1111x_1^2(t_2x_3+t_1x_4)-99x_2^2(t_2x_3+t_1x_4)=0. \quad \quad \quad \quad\quad \quad
\end{multline*}
Then $X$ is a smooth Fano 3-fold in the~family \textnumero 3.2. This family is denoted by $T_{11}$ in \cite{ChPrSh}.
Note that the~natural projection $V\to\mathbb{P}^1$ induces a morphism $\pi\colon X\to\mathbb{P}^1_{t_1,t_2}$ whose general fiber is a smooth cubic surface in $\mathbb{P}^3$.
Let $A$ be the~subgroup in $\mathrm{Aut}(V)$ that is generated by the~following transformations:
\begin{align*}
(t_1,t_2;x_1,x_2,x_3,x_4)&\mapsto(t_2,t_1;x_1,x_2,x_4,x_3),\\
(t_1,t_2;x_1,x_2,x_3,x_4)&\mapsto(t_1,-t_2;x_1,-x_2,x_3,-x_4).
\end{align*}
Then $A\simeq(\mathbb{Z}/2\mathbb{Z})^2$, the~3-fold $X$ is $A$-invariant, the~group $A$ acts faithfully on $X$,
and the~morphism $\pi$ is $A$-equivariant. Moreover, the~group $A$ does not fix points in $\mathbb{P}^1_{t_1,t_2}$, so that it does not fix points in $X$.
\end{example}

\begin{example}
\label{example:3-5}
Let $S=\mathbb{P}^1_{u,v}\times\mathbb{P}^1_{x,y}$, and let $A$ be the subgroup in $\mathrm{Aut}(S)$ generated by
\begin{align*}
([u:v],([x:y])&\mapsto([v:u],[y:x]),\\
([u:v],([x:y])&\mapsto([-u:v],[-x:y]).
\end{align*}
Then it follows from \cite[Lemma~A.54]{Book} that $S$ contains an $A$-invariant smooth curve $C$ of degree $(5,1)$.
Now, we consider the~$A$-equivariant embedding $S\hookrightarrow\mathbb{P}^1\times\mathbb{P}^2$ given by
$$
\big([u:v],[x:y]\big)\mapsto\big([u:v],[x^2:xy:y^2]\big).
$$
We identify $S$ and $C$ with their images in $\mathbb{P}^1\times\mathbb{P}^2$,
and we identify $A$ with a~subgroup in $\mathrm{Aut}(\mathbb{P}^1\times\mathbb{P}^2)$.
Then $A$ does not fix points in $\mathbb{P}^1\times\mathbb{P}^2$.
Let $\pi\colon X\to\mathbb{P}^1\times\mathbb{P}^2$ be the~blow up along $C$.
Then $X$ is a~smooth Fano 3-fold in Family~\textnumero~3.5,
the~action of the~group $A$ lifts to $X$, and $A$ does not fix points in $X$.
\end{example}

\begin{example}
\label{example:3-6}
Let $L$ be the~line $\{x_1-x_3=x_2-x_4=0\}\subset\mathbb{P}^3$,
let $C$ be the~curve
$$
\big\{x_1^2+x_2^2+2025(x_3^2+x_4^2)=0, 2025(x_1^2-x_2^2)+x_3^2-x_4^2=0\big\}.
$$
Then $L\cap C=\varnothing$, and $C$ is a smooth elliptic curve.
Let $\pi\colon X\to\mathbb{P}^3$ be the~blow up along $L$ and $C$.
Then $X$ is a~smooth Fano 3-fold in Family \textnumero 3.6.
Let $A$ be the~subgroup in $\mathrm{Aut}(\mathbb{P}^3)$  generated by
\begin{align*}
[x_1:x_2:x_3:x_4]&\mapsto[x_2:x_1:x_4:x_3],\\
[x_1:x_2:x_3:x_4]&\mapsto[x_1:-x_2:x_3:-x_4].
\end{align*}
Then $A\simeq(\mathbb{Z}/2\mathbb{Z})^2$, and both curves $L$ and $C$ are $A$-invariant, so the~action of the~group $A$ lifts to $X$.
Moreover, the~group $A$ does not fix points in $\mathbb{P}^3$, so that it does not fix points in $X$.
\end{example}

\begin{example}
\label{example:3-7}
Let $A$ be the~subgroup in $\mathrm{Aut}(\mathbb{P}^2\times\mathbb{P}^2)$ that is generated by the~following transformations:
\begin{align*}
([x_1:x_2:x_3],[y_1:y_2:y_3])&\mapsto([x_2:x_3:x_1],[y_2:y_3:y_1]),\\
([x_1:x_2:x_3],[y_1:y_2:y_3])&\mapsto([\zeta_3^2x_1:\zeta_3x_2:x_3],[\zeta_3y_1:\zeta_3^2y_2:y_3]),
\end{align*}
where $\zeta_3$ is a primitive cube root of unity. Then $A\simeq(\mathbb{Z}/3\mathbb{Z})^2$, and $A$ acts on $\mathbb{P}^2\times\mathbb{P}^2$ without fixed points.
Let $W$, $W^\prime$, $W^{\prime\prime}$ be divisors of degree $(1,1)$ in $\mathbb{P}^2\times\mathbb{P}^2$ that are given by
\begin{gather*}
x_1y_1+x_2y_2+x_3y_3=0,\\
\zeta_3x_1y_1+\zeta_3^2x_2y_2+x_3y_3=0,\\
x_1y_2+x_2y_3+x_3y_1=0,
\end{gather*}
respectively. Then they are smooth and $A$-invariant. Set $C=W\cap W^\prime\cap W^{\prime\prime}$. Then $C$ is a smooth elliptic curve.
Let $\pi\colon X\to W$ be the~blowup of the~curve $C$.
Then $X$ is a smooth Fano 3-fold in Family \textnumero 3.7, and the~action of the~group $A$ lifts to $X$,
and $A$ does not fix points in $X$.
\end{example}

\begin{example}
\label{example:3-10}
Let $Q$ be the~smooth quadric 3-fold in $\mathbb{P}^4$ that is given by $x_1^2+x_2^2+x_3^2+x_4^2+x_5^2=0$,
let $\Pi_1=\{x_1=x_2=0\}$ and $\Pi_2=\{x_2=x_4=0\}$,
and let $A$ be the~subgroup in $\mathrm{Aut}(Q)$ generated by
\begin{align*}
[x_1:x_2:x_3:x_4:x_5]&\mapsto [-x_1:x_2:x_3:x_4:x_5],\\
[x_1:x_2:x_3:x_4:x_5]&\mapsto [x_1:-x_2:x_3:x_4:x_5],\\
[x_1:x_2:x_3:x_4:x_5]&\mapsto [x_1:x_2:-x_3:x_4:x_5],\\
[x_1:x_2:x_3:x_4:x_5]&\mapsto [x_1:x_2:x_3:-x_4:x_5].
\end{align*}
Then $A\simeq (\mathbb{Z}/2\mathbb{Z})^4$, and $A$ does not have fix points on $Q$,
planes $\Pi_1$ and $\Pi_2$ are $A$-invariant.
Let $\pi\colon X\to Q$ be the~blowup of the~disjoint conics $Q\cap\Pi_1$ and $Q\cap\Pi_2$.
Then $X$ is a smooth Fano 3-fold in Family \textnumero 3.10, the~action of the~group $A$ lifts to $X$,
and $A$ does not fix points in~$X$.
\end{example}

\begin{example}
\label{example:3-12}
Let $L$ be the line $\{x_1=x_4=0\}\subset\mathbb{P}^3$, let $C$ be the twisted cubic $\varphi(\mathbb{P}^1)$ for $\varphi\colon \mathbb{P}^1\to \mathbb{P}^3$ given by $[u:v]\mapsto [v^3:uv^2:u^2v:u^3]$.
Then $L\cap C=\varnothing$.
Let $\pi\colon X\to\mathbb{P}^3$ be the~blowup along $L$ and $C$.
Then $X$ is a smooth Fano 3-fold in Family \textnumero 3.12.
Let $A$ be the~subgroup in $\mathrm{Aut}(\mathbb{P}^3)$ generated by
\begin{align*}
[x_1:x_2:x_3:x_4]&\mapsto[x_4:x_3:x_2:x_1],\\
[x_1:x_2:x_3:x_4]&\mapsto[x_1:-x_2:x_3:-x_4].
\end{align*}
Then $A\simeq(\mathbb{Z}/2\mathbb{Z})^2$, both curves $L$ and $C$ are $A$-invariant, so the~action of the~group $A$ lifts to $X$.
Moreover, the~group $A$ does not fix points in $\mathbb{P}^3$, hence it does not fix points in $X$.
\end{example}

\begin{example}
\label{example:3-13}
In $\mathbb{P}^2_{x_1,x_2,x_3}\times\mathbb{P}^2_{y_1,y_2,y_3}\times\mathbb{P}^2_{z_1,z_2,z_3}$, we let
$$
X=\big\{x_1y_1+x_2y_2+x_3y_3=0,x_1z_1+x_2z_2+x_3z_3=0,y_1z_1+y_2z_2+y_3z_3=0\big\}.
$$
Then $X$ is a smooth Fano 3-fold in Family \textnumero 3.13, and $\mathrm{Aut}(X)\simeq\mathrm{PGL}_2(\mathbb{C})\times\mathfrak{S}_3$ by \cite[Remark~5.100]{Book}.
Let $A$ be the subgroup in $\mathrm{Aut}(X)$ generated by
\begin{align*}
\big([x_1:x_2:x_3],[y_1:y_2:y_3],[z_1:z_2:z_3]\big)&\mapsto \big([-x_1:x_2:x_3],[-y_1:y_2:y_3],[-z_1:z_2:z_3]\big),\\
\big([x_1:x_2:x_3],[y_1:y_2:y_3],[z_1:z_2:z_3]\big)&\mapsto \big([x_1:-x_2:x_3],[y_1:-y_2:y_3],[z_1:-z_2:z_3]\big),\\
\big([x_1:x_2:x_3],[y_1:y_2:y_3],[z_1:z_2:z_3]\big)&\mapsto \big([y_1:y_2:y_3],[x_1:x_2:x_3],[z_1:z_2:z_3]\big).
\end{align*}
Then $A\simeq(\mathbb{Z}/2\mathbb{Z})^3$, and $A$ does not fix points in $X$.
\end{example}

\begin{example}[{\cite[\S~5.22]{Book}}]
\label{example:4-13}
Let $V=\mathbb{P}^1_{x_1,x_2}\times\mathbb{P}^1_{y_1,y_2}\times\mathbb{P}^1_{z_1,z_2}$, and let $A$ be the~subgroup of $\operatorname{Aut}(V)$ generated by the~transformations
\begin{align*}
\big([x_1 :x_2],[y_1 :y_2],[z_1 :z_2]\big)&\mapsto \big([x_1 :-x_2],[y_1 :-y_2],[z_1 :-z_2]\big),\\
\big([x_1 :x_2],[y_1 :y_2],[z_1 :z_2]\big)&\mapsto \big([x_2 :x_1],[y_2 :y_1],[z_2 :z_1]\big), \\
\big([x_1 :x_2],[y_1 :y_2],[z_1 :z_2]\big)&\mapsto \big([y_1 :y_2],[x_1 :x_2],[z_1 :z_2]\big).
\end{align*}
Then $A\simeq (\mathbb{Z}/2\mathbb{Z})^3$, and $A$ does not fix points in $V$.
Let  $C$ be the curve of degree $(1,1,3)$ in $V$ given by
$$
x_1y_2 - x_2y_1 = x_1^3z_1 + x_2^3z_2 + \lambda(x_1x_2^2z_1 + x_1^2x_2z_2) = 0,
$$
where $\lambda\in\mathbb{C}\setminus\{\pm 1,\pm 3\}$.
Then $C$ is smooth and $A$-invariant.
Let $\pi\colon X \to V$ be the~blowup of $V$ along $C$.
Then $X$ is a smooth Fano 3-fold in Family \textnumero 4.13,
and it follows from  \cite[\S~5.22]{Book} that every K-polystable smooth Fano 3-fold in Family \textnumero 4.13 can be obtained in this way.
Since $C$ is $A$-invariant, the~action of the~group $A$ lifts to $X$,
and $S$ does not fix points in~$X$.
\end{example}

\medskip
\noindent
\textbf{Acknowledgements.}
We thank Michel Brion, Paolo Cascini, Alexander Duncan and Pham Tiep for fruitful discussions. This work bloomed during a two-week CIRM RIR program in Luminy which was part of a semester-long \emph{Morlet Chair}. We thank the~staff of CIRM for facilitating high quality working environments for research.
Ivan Cheltsov has been supported by Simons Collaboration grant \emph{Moduli of Varieties}. Hamid Abban has been supported by EPSRC grant EP/Y033450/1 and a Royal Society International Collaboration Award ICA$\backslash$1$\backslash$231019. Takashi Kishimoto has been supported by JSPS KAKENHI Grant Number 23K03047.


\begin{thebibliography}{33}

\bibitem{AbbanCheltsovKishimotoMangolte}
H.~Abban, I.~Cheltsov, T.~Kishimoto, F.~Mangolte, \emph{K-stability of pointless del Pezzo surfaces and Fano 3-folds}, preprint, arXiv:2411.00767, 2024.

\bibitem{AbbanZhuang}
H.~Abban, Z.~Zhuang, \emph{K-stability of Fano varieties via admissible flags},
Forum of Mathematics Pi \textbf{10} (2022), 1--43.

\bibitem{AbbanZhuangSeshadri}
H.~Abban, Z.~Zhuang, \emph{Seshadri constants and K-stability of Fano manifolds}, Duke Mathematical Journal  \textbf{172} (2023) 1109--1144.

\bibitem{Book}
C.~Araujo, A.-M.~Castravet, I.~Cheltsov, K.~Fujita, A.-S.~Kaloghiros, J.~Martinez-Garcia, C.~Shramov, H.~S\"u\ss, N.~Viswanathan, \emph{The Calabi problem for Fano 3-folds}, Cambridge University Press, \textbf{485} (2023).

\bibitem{BelousovLoginov}
G.~Belousov, K.~Loginov, \emph{K-stability of Fano 3-folds of rank $4$ and degree $24$},
Eur. J. Math. \textbf{9} (2023), Paper No. 80, 20 p.

\bibitem{BelousovLoginov2024}
G.~Belousov, K.~Loginov, \emph{K-stability of Fano 3-folds of rank $3$ and degree $14$}, preprint, arXiv:2403.03700 (2024).

\bibitem{Cheltsov2008}
I.~Cheltsov, \emph{Log canonical thresholds of del Pezzo surfaces}, Geometric and Functional Analysis \textbf{11} (2008), 1118--1144.

\bibitem{Cheltsov2024}
I.~Cheltsov, \emph{K-stability of Fano 3-folds of Picard rank $3$ and degree $22$}, Sao Paulo Journal of Mathematical Sciences \textbf{18} (2024), 609--638.

\bibitem{CheltsovDenisovaFujita}
I.~Cheltsov, E.~Denisova, K.~Fujita, \emph{K-stable smooth Fano 3-folds of Picard rank two}, Forum Math. Sigma \textbf{12} (2024), Paper No. e41, 64 p.

\bibitem{CheltsovFedorchukFujitaKaloghiros}
I.~Cheltsov, M.~Fedorchuk, K.~Fujita, A.-S.~Kaloghiros, \emph{K-moduli of pure states of four qubits}, preprint, arXiv:2412.19972, 2024.

\bibitem{CheltsovFujitaKishimotoOkada}
I.~Cheltsov, K.~Fujita, T.~Kishimoto, T.~Okada, \emph{K-stable divisors in $\mathbb{P}^1\times\mathbb{P}^1\times\mathbb{P}^2$ of degree $(1,1,2)$},
Nagoya Mathematical Journal \textbf{251} (2023), 686--714.

\bibitem{CheltsovFujitaKishimotoPark}
I.~Cheltsov, K.~Fujita, T.~Kishimoto, J.~Park, \emph{K-stable Fano 3-folds in the~families 2.18 and 3.4}, preprint, arXiv:2304.11334, 2023.

\bibitem{CheltsovLiMauPinardin}
I.~Cheltsov, O.~Li, S.~Ma'u, A.~Pinardin, \emph{K-stability and space curves of degree $6$ and genus $3$}, preprint, arXiv:2404.07803, 2024.

\bibitem{ChPrSh}
I.~Cheltsov, V.~Przyjalkowski, C.~Shramov, \emph{Hyperelliptic and trigonal Fano 3-folds}, Izv. Math. \textbf{69} (2005), 365--421.

\bibitem{CheltsovPrzyjalkowskiShramov}
I.~Cheltsov, V.~Przyjalkowski, C.~Shramov, \emph{Fano 3-folds with infinite automorphism groups},
Izvestia: Math. \textbf{83} (2019), 860--907.

\bibitem{CheltsovShramov2015}
I.~Cheltsov, C.~Shramov, \emph{Cremona groups and the~icosahedron}, CRC press, 2015.

\bibitem{CheltsovTschinkelZhang2024}
I.~Cheltsov, Yu.~Tschinkel, Zh.~Zhang, \emph{Equivariant geometry of singular cubic threefolds}, Forum Math. Sigma, to appear.

\bibitem{CheltsovTschinkelZhang2025}
I.~Cheltsov, Yu.~Tschinkel, Zh.~Zhang, \emph{Equivariant unirationality of Fano threefolds}, preprint, arXiv:2502.19598, 2025.

\bibitem{Corti2000}
A.~Corti, \emph{Singularities of linear systems and 3-fold birational geometry}, London Mathematical
Society Lecture Note Series \textbf{281} (2000), 259--312.

\bibitem{Denisova}
E.~Denisova, \emph{On K-stability of $\mathbb{P}^3$ blown up along the~disjoint union of a~twisted cubic curve and~a~line}, Journal of the~London Math. Society \textbf{109} (2024), Paper e12911.

\bibitem{Dolgachev1999}
I.~Dolgachev, \emph{Invariant stable bundles over modular curves $X(p)$}, in: Recent
Progress in Algebra (Taejon/Seoul, 1997), Contemporary Mathematics \textbf{224},
Amer. Math. Soc., Providence, RI, 1999, 65--99.

\bibitem{FantechiPardini}
B.~Fantechi, R.~Pardini, \emph{Automorphisms and moduli spaces of varieties with ample canonical class via deformations of abelian covers},
Communications in Algebra \textbf{25} (1997), 1413--1441.

\bibitem{Fujita2016}
K.~Fujita, \emph{On K-stability and the~volume functions of $\mathbb{Q}$-Fano varieties}, Proc. Lond. Math. Soc. \textbf{113} (2016), 541--582.

\bibitem{Fujita2019}
K.~Fujita, \emph{A valuative criterion for uniform K-stability of $\mathbb{Q}$-Fano varieties},  J. Reine Angew. Math. \textbf{751} (2019), 309--338.

\bibitem{LTN}
L.~Giovenzana, T.~Duarte Guerreiro, N.~Viswanathan, \emph{On K-stability of $\mathbb{P}^3$ blown up along a $(2,3)$ complete intersection}, Journal of London Math. Society \textbf{110} (2024), Paper e12961.

\bibitem{Alvaro}
V.~Gonzalez-Aguilera, A.~Liendo, P.~Montero,
\emph{On the~liftability of the~automorphism group of smooth hypersurfaces of the~projective space}, Isr. J. Math. \textbf{255} (2023), 283--310.

\bibitem{Iliev}
A.~Iliev, \emph{Log-terminal singularities of algebraic surfaces}, Mosc. Univ. Math. Bull. \textbf{41} (1986), 46--53.

\bibitem{IsPr99}
V.~Iskovskikh, Yu.~Prokhorov, \emph{Fano varieties},  Encyclopaedia of Mathematical Sciences \textbf{47} (1999) Springer, Berlin.

\bibitem{Kawamata1998}
Y. Kawamata, \emph{Subadjunction of log canonical divisors II}, American J. Math. \textbf{120} (1998), 893--899.

\bibitem{Knutsen}
A.~Knutsen, \emph{A note on Seshadri constants on general K3 surfaces}, C. R., Math., Acad. Sci. Paris \textbf{346}, (2008), 1079--1081.

\bibitem{Kollar-Olaf}
J.~Koll\'ar, F.-O.~Schreyer, \emph{Real Fano 3-folds of type $V_{22}$}, The Fano conference, 515--531, Univ. Torino, Turin, 2004.

\bibitem{LGBT}
The LMFDB Collaboration, \emph{The L-functions and modular forms database}, \texttt{https://www.lmfdb.org}, 2025.

\bibitem{KuznetsovProkhorovShramov}
A.~Kuznetsov, Yu.~Prokhorov, C.~Shramov, \emph{Hilbert schemes of lines and conics and automorphism groups of Fano 3-folds},
Japanese Journal of Mathematics \textbf{13} (2018), 109--185.

\bibitem{Lazarsfeld}
R.~Lazarsfeld, \emph{Positivity in Algebraic Geometry II}, Ergeb. Math. Grenzgeb. \textbf{49}, Springer, 2004.

\bibitem{Li2017}
C.~Li, \emph{K-semistability is equivariant volume minimization},  Duke Math. J. \textbf{166} (2017), 3147--3218.

\bibitem{Liu}
Y.~Liu, \emph{K-stability of Fano 3-folds of rank $2$ and degree $14$ as double covers}, to appear in Mathematische Zeitschrift.

\bibitem{LXZ}
Y.~Liu, C.~Xu, Z.~Zhuang, \emph{Finite generation for valuations computing stability thresholds and applications to  {K}-stability}, Ann. of Math. \textbf{196} (2022), 507--566.

\bibitem{Malbon}
J.~Malbon, \emph{Automorphisms of Fano threefolds of rank $2$ and degree $28$}, Annalli dell'Universita di Ferrara \textbf{70} (2024), 1083--1092.

\bibitem{Malbon2024}
J.~Malbon, \emph{K-stable Fano 3-folds of rank $2$ and degree $28$}, preprint, arXiv:2304.12295, 2023.

\bibitem{Matsuki}
K.~Matsuki, \emph{Weyl groups and birational transformations among minimal models}, Mem. Am. Math. Soc. \textbf{557} (1995), 133 pages.

\bibitem{Nadel}
A.~Nadel, \emph{Multiplier ideal sheaves and K\"ahler--Einstein metrics of positive scalar curvature}, Annals of Mathematics \textbf{132} (1990), 549--596.

\bibitem{Namikawa97}
Y.~Namikawa, \emph{Smoothing Fano 3-folds}, J. Algebr. Geom. \textbf{6} (1997), 307--324.

\bibitem{Prokhorov2013}
Yu.~Prokhorov, \emph{$G$-Fano 3-folds. II}, Adv. Geom. \textbf{13} (2013), 419--434.

\bibitem{ReYou00}
Z.~Reichstein, B.~Youssin, \emph{Essential dimensions of algebraic groups and a resolution theorem for $G$-varieties},
with an appendix by J.~Koll\'ar and E.~Szabo, Canad. J. Math. \textbf{52} (2000), no. 5, 1018--1056.

\bibitem{Reid}
M.~Reid, \emph{Chapters on algebraic surfaces}, IAS/Park City Math. Ser. \textbf{3} (1997), 3--159.

\bibitem{Rim}
D.~Rim, \emph{Equivariant $G$-structure on versal deformations}, Transactions of the~American Mathematical Society \textbf{257} (1980), 217--226.

\bibitem{Schreyer}
F.-O.~Schreyer, \emph{Geometry and algebra of prime Fano 3-folds of genus 12}, Compos. Math. \textbf{127} (2001), 297--319.

\bibitem{Zhao}
J.~Zhao, \emph{K-stability of Thaddeus' moduli of stable bundle pairs on genus two curves}, preprint, arXiv:2412.18064, 2024.

\bibitem{Zhuang2021}
Z.~Zhuang, \emph{Optimal destabilizing centers and equivariant K-stability}, Invent. Math. \textbf{226} (2021), \mbox{195--223}.
\end{thebibliography}
\end{document}